\documentclass[11pt, a4paper]{amsart}
\usepackage{amssymb, array, amsmath, amscd, pdfpages, enumerate, amsthm, fixltx2e, setspace, mathtools}

\usepackage[utf8]{inputenc}
\usepackage{amssymb}
\usepackage{bm}
\usepackage{mathrsfs}
\usepackage{graphicx} 
\usepackage[bottom=1.3in,hmarginratio=1:1]{geometry}

\newcommand{\mc}[1]{\mathcal{#1}}

\newcommand{\eps}{\varepsilon}

\DeclareMathOperator{\re}{Re}

\DeclareMathOperator{\Span}{span}
\DeclareMathOperator{\diag}{diag}

\newcommand*{\C}{{\mathbb{C}}}     
\newcommand*{\R}{{\mathbb{R}}}     
\newcommand*{\Z}{{\mathbb{Z}}}     
     
\newcommand*{\N}{{\mathbb{N}}}

\newcommand*{\Lin}{{\mathcal{L}}}   
\newcommand*{\Dom}{{\mathcal{D}}}   
\newcommand{\ran}{{\mathcal{R}}}   
\renewcommand{\ker}{{\mathcal{N}}}

\newcommand*{\abs} [1]{\lvert#1\rvert}
\newcommand*{\norm}[1]{\lVert#1\rVert}
\newcommand*{\set} [1]{\{#1\}}
\newcommand*{\setm}[2]{\{\,#1\mid#2\,\}}   
\newcommand*{\iprod}[2]{\langle#1,#2\rangle}

\newcommand*{\Setm}[2]{\left\{\,#1\,\middle|\,#2\,\right\}}
\newcommand*{\Lp}[1][p]{L^{#1}}

\newcommand*{\Lploc}[1][p]{L^{#1}_{\text{loc}}}
\newcommand*{\lp}[1][p]{\ell^{#1}} 

\newcommand{\pmat}[1]{\begin{pmatrix}#1\end{pmatrix}}
\newcommand{\pmatsmall}[1]{\begin{psmallmatrix}#1\end{psmallmatrix}}

\newcommand*{\Abs}[2][default]{\ifthenelse{\equal{#1}{default}}{\left\lvert#2\right\rvert}{\ldelim{#1}{\lvert}#2\rdelim{#1}{\rvert}}}
\newcommand*{\Norm}[2][default]{\ifthenelse{\equal{#1}{default}}{\left\lVert#2\right\rVert}{\ldelim{#1}{\lVert}#2\rdelim{#1}{\rVert}}}

\newcommand*{\Iprod}[3][default]{\ifthenelse{\equal{#1}{default}}{\left\langle#2,#3\right\rangle}{\ldelim{#1}{\langle}#2,#3\rdelim{#1}{\rangle}}}
\newcommand*{\Dualpair}[3][default]{\ifthenelse{\equal{#1}{default}}{\left\langle#2,#3\right\rangle}{\ldelim{#1}{\langle}#2,#3\rdelim{#1}{\rangle}}}

\newcommand{\eq}[1]{\begin{align*}#1\end{align*}}
\newcommand{\eqn}[1]{\begin{align}#1\end{align}}

\newcommand{\gs}{\sigma}
\newcommand{\ga}{\alpha}
\newcommand{\gb}{\beta}
\renewcommand{\gg}{\gamma}
\newcommand{\gd}{\delta}
\newcommand{\gl}{\lambda}
\newcommand{\gw}{\omega}

\newcommand{\ieq}[1]{$#1$}

\newcommand{\inv}{^{-1}}

\newcommand*{\ddb}[2][1]{\ifthenelse{\equal{#1}{1}}{\frac{d}{d#2}}{\frac{d^{#1}}{d#2^{#1}}}}
\newcommand*{\pd}[3][1]{\ifthenelse{\equal{#1}{1}}{\frac{\partial{#2}}{\partial{#3}}}{\frac{\partial^{#1}{#2}}{\partial#3^{#1}}}}

\newcommand*{\keyterm}[1]{\emph{#1}}

\newtheorem{thm}{Theorem}

\newtheorem{lem}[thm]{Lemma}
\newtheorem{cor}[thm]{Corollary}
\theoremstyle{definition}
\newtheorem{dfn}[thm]{Definition}
\newtheorem{ass}[thm]{Assumption}
\newtheorem{rem}[thm]{Remark}

\newtheorem*{ORP}{The Output Regulation Problem}
\newtheorem*{RORP}{The Robust Output Regulation Problem}

\newcommand{\Gconds}{$\mc{G}$-conditions}
\newcommand{\yref}{y_{\mbox{\scriptsize\textit{ref}}}}

\newcommand{\Omi}{\mathcal{O}}

\newcommand{\Ops}{\mc{O}}
\newcommand*{\pinv}{^{\dagger}}

\newcommand{\CL}{C_\Lambda }
\newcommand{\KL}{K_\Lambda }
\newcommand{\CLt}{\tilde{C}_\Lambda }
\newcommand{\CeL}{C_{e\Lambda} }
\newcommand{\CeLt}{\tilde{C}_{e\Lambda} }
\newcommand{\XB}{X_B }
\newcommand{\XBd}{X_{B_d} }
\newcommand{\XBBd}{X_{(B,B_d)} }
\newcommand{\XBBdt}{X_{(\tilde{B},\tilde{B}_d)} }
\newcommand{\XBL}{X_{(B,L)} }
\newcommand{\XL}{X_L }

\newcommand{\ZG}{Z_{\mc{G}_2} }

\newcommand{\Amo}{A}
\newcommand{\Atmo}{\tilde{A}}
\newcommand{\Tmo}{T}
\newcommand{\Temo}{T_e}

\newcommand{\Aemo}{A_e}
\newcommand{\Aetmo}{\tilde{A}_e}
\newcommand{\Gmo}{\mc{G}_1}

\newcommand{\xeoo}{\tilde{x}_{e0}}

\newcommand{\KoL}{K_1}
\newcommand{\KtL}{K_2^\Lambda}
\newcommand{\KtoL}{K_{21}^\Lambda}

\newcommand{\Bw}{B_d}
\newcommand{\Uw}{U_d}

\newcommand{\BI}{G_2 }

\newcommand{\BIn}[1][n]{G_{2#1}}

\newcommand{\Wf}{W}

\newcommand{\Mlog}{M_{\textup{log}}}

\newcommand{\Jinds}{J}

\newcommand{\ZI}{Z_0}

\renewcommand{\v}{v}

\newcommand{\RBm}[1][k]{R(i\gw_{#1},\Aemo )B_e\phi_{#1}}

\newcommand{\PARsysopspert}{(\tilde{A},\tilde{B},\tilde{B}_d,\tilde{C},\tilde{D},\tilde{E},\tilde{F})}
\newcommand{\PARsysops}{(A,B,B_d,C,D,E,F)}

\newcommand{\CLops}{(A_e,B_e,C_e,D_e)}

\newcommand{\PARcontr}{(\mc{G}_1,\mc{G}_2,K)}

\newcommand{\kZ}{_{k\in\Z}}

\renewcommand{\pmat}[1]{\begin{bmatrix}#1\end{bmatrix}}
\renewcommand{\pmatsmall}[1]{\begin{bsmallmatrix}#1\end{bsmallmatrix}}

\begin{document}

\title[Robust Controllers for Regular Linear Systems]{Robust Controllers for Regular Linear Systems with Infinite-Dimensional Exosystems}

\thispagestyle{plain}

\author{Lassi Paunonen}
\address{Department of Mathematics, Tampere University of Technology, PO.\ Box 553, 33101 Tampere, Finland}
\email{lassi.paunonen@tut.fi}

\begin{abstract}
We construct two error feedback controllers for robust output tracking and disturbance rejection of a regular linear system with nonsmooth reference and disturbance signals. We show that for sufficiently smooth signals the output converges to the reference at a rate that depends on the behaviour of the transfer function of the plant on the imaginary axis. In addition, we construct a controller that can be designed to achieve robustness with respect to a given class of uncertainties in the system, and present a novel controller structure for output tracking and disturbance rejection without the robustness requirement. We also generalize the internal model principle for regular linear systems
with boundary disturbance and for controllers with unbounded input and output operators. The construction of controllers is illustrated with an example where we consider output tracking of a nonsmooth periodic reference signal for a two-dimensional heat equation with boundary control and observation, and with periodic disturbances on the boundary.  
\end{abstract}

\thanks{The research was supported by the Academy of Finland grant number 298182.}

\maketitle

\section{Introduction}
\label{sec:intro}

The purpose of this paper is to construct controllers for robust output regulation 
of a \keyterm{regular linear system}\footnote{Here $\CL$ and $\KL$ denote the $\Lambda$-extensions of $C$ and $K$, respectively. See Section~\ref{sec:plantexo} for details.}~\cite{Wei89b,Wei94,StaWei02}%
\begin{subequations}
  \label{eq:plantintro}
  \eqn{
  \dot{x}(t)&= Ax(t)+Bu(t) + \Bw w(t), \qquad x(0)=x_0\in X\\
  y(t)& = \CL x(t) + Du(t)
  }
\end{subequations}%
on an infinite-dimensional Banach space $X$.
The main goal in the control problem is to achieve asymptotic convergence of the output $y(t)$
to a given
reference signal $\yref(t)$ despite external disturbance signals $w(t)$. In addition, it is required that the controller is robust in the sense that output tracking is achieved even under perturbations and uncertainties in the operators $(A,B,B_d,C,D)$ of the plant. 
The class of regular linear systems 
facilitates the study of robust output tracking and disturbance rejection for 
many important classes partial differential equations with boundary control and observation 
with corresponding
 unbounded operators $B$, $B_d$ and $C$~\cite{ByrGil02,GuoZha07,ZwaLeG10,NatGil14}.
In this paper we continue the work on designing robust controllers for regular linear systems begun recently in~\cite{Pau16a}.

The reference signal $\yref(t)$ and the disturbance signals $w(t)$ considered in the
robust output regulation 
problem are assumed to be 
generated by an \keyterm{exosystem} of the form%
\begin{subequations}
  \label{eq:exointro}
  \eqn{
  \dot{v}(t)&=Sv(t), \qquad v(0)=v_0\in W\\
  w(t)&=Ev(t)\\
  \yref(t)&=-Fv(t).
  }
\end{subequations}%
In the case where the exosystem~\eqref{eq:exointro} is a system of ordinary differential equations on a finite-dimensional space $W$, the class of reference and disturbance signals consists
of finite linear combinations of trigonometric functions and polynomially increasing terms.  
In this paper we concentrate on 
output tracking and disturbance rejection for 
a general class of
nonsmooth periodic and almost periodic reference and disturbance signals. Such exogeneous signals can be generated with infinite-dimensional exosystems on the Hilbert space $W=\lp[2](\C)$ where $S=\diag(i\gw_k)_{k\in\Z}$ is an unbounded diagonal operator containing the frequencies $\set{\gw_k}_{k\in\Z}$ 
that are present in
the signals $\yref(\cdot)$ and $w(\cdot)$.
In particular, any continuous $\tau$-periodic signal can be generated with an exosystem of the form~\eqref{eq:exointro} where
$S=\diag \bigl(i \frac{2k\pi}{\tau}\bigr)_{k\in\Z}$~\cite{ImmPoh05b}.
Output tracking and disturbance rejection of nonsmooth signals with high accuracy have applications in the control of motor and disk drive systems and in power electronics~\cite{CosGri09}. 
Output tracking of
signals generated by an infinite-dimensional exosystem have been studied using state space methods in~\cite{ImmPoh06b,HamPoh10,PauPoh10,NatGil14,PauPoh14a}, and using frequency domain techniques in~\cite{YamHar88,RebWei03,YliPoh06,LaaPoh15}. 
Robust tracking of nonsmooth periodic functions has also been studied extensively in~\keyterm{repetitive control}~\cite{HarYam88,Yam93,WeiHaf99} where the control objective is to achieve precise tracking for a finite number of frequency components of $\yref(\cdot)$.

As the main results of the paper we introduce two methods for constructing a regular error feedback controller of the form%
\begin{subequations}%
  \label{eq:controllerintro}
  \eqn{
  \dot{z}(t) &= \mc{G}_1z(t)+\mc{G}_2(y(t)-\yref(t)), \qquad z(0)=z_0\in Z\\
  u(t) &= \KL z(t)
  }%
\end{subequations}%
to achieve robust output tracking and disturbance rejection for the regular linear system~\eqref{eq:plantintro}. 
The \keyterm{internal model principle} 
of linear control theory
states that in order to solve the robust output regulation problem it is both necessary and sufficient for the feedback controller~\eqref{eq:controllerintro} to include a suitable number of independent copies of the dynamics of the exosystem~\eqref{eq:exointro} and to achieve closed-loop stability.
This fundamental characterization of robust controllers was originally presented for finite-dimensional linear systems by Francis and Wonham~\cite{FraWon75a} and Davison~\cite{Dav76} in 1970's, and it was later generalized for infinite-dimensional linear systems with finite and infinite-dimensional exosystems in~\cite{PauPoh10,PauPoh14a}.
The internal model principle also implies that
the robust controllers always tolerate a class of uncertainties and inaccuracies in the parameters $\mc{G}_2$ and $K$ and in certain parts of the operator $\mc{G}_1$
of the controller~\eqref{eq:controllerintro}. This property can be exploited in controller design as it 
sometimes
allows the use of approximations in defining $\mc{G}_1$, $\mc{G}_2$, and $K$, provided that the internal model property is preserved and the closed-loop system achieves the necessary stability properties.

The two robust controllers constructed in this paper utilize two different internal model based structures that are naturally complementary to each other. The first construction is based on a new block triangular controller structure that was first introduced in~\cite{Pau15a,Pau16a} for control of regular linear systems with finite-dimensional exosystems. The second controller uses the observer based structure that was used to solve the robust output regulation problem for an infinite-dimensional exosystem in~\cite{HamPoh10} in the case where the plant had bounded input and output operators.
In this paper we generalize both of the controller structures to accommodate for an infinite-dimensional internal model and for unbounded operators $B$, $B_d$, and $C$ in the plant. In particular, this requires the use of new techniques in the analysis of the well-posedness and stability of the resulting closed-loop systems. The constructions we present allow the use of unbounded feedback and output injection operators in achieving exponential stability of the pairs $(A,B)$ and $(C,A)$.  

The two controllers that we construct end up possessing slightly differing properties.  
In the case of an infinite-dimensional exosystem the reference and disturbance signals are required to have a certain minimum level of smoothness in order for the robust output regulation problem to be solvable, and the exact level 
depends on the behaviour of the transfer function $P(\gl) = \CL R(\gl,A)B+D$ of the plant at the frequencies $\gl=i\gw_k$ of the exosystem~\cite{ImmPoh06a,HamPoh10}. 
More precisely, 
faster growth of the norms $\norm{P(i\gw_k)\pinv}$ of the Moore--Penrose pseudoinverses of $P(i\gw_k)$ as $\abs{k}\to \infty$ leads to a higher minimal level of smoothness for $\yref(t)$ and $w(t)$.  
Our results demonstrate that the required level of smoothness is in general lower in the case of the new controller structure than in the case of the observer based controller structure.  
Moreover, the new controller structure can be used in a situation where the plant has a larger number of inputs than outputs, whereas the construction of the second observer based controller requires that the input and output spaces of the plant are isomorphic.

Since the output operator $C$ of the plant is in general unbounded, the regulation error $e(t)=y(t)-\yref(t)$ is not guaranteed to converge to zero as $t\to \infty$.
In this paper the convergence of the regulation error is instead considered in the sense that
\eqn{
\label{eq:regerrdecintro}
\int_t^{t+1} \norm{y(s)-\yref(s)}ds\to 0, \qquad \mbox{as} ~ t\to \infty.
}
This condition is equivalent to requiring that the averages $\frac{1}{\eps}\int_t^{t+\eps}\norm{e(s)}ds$ of the error over the intervals $[t,t+\eps]$ for any fixed $\eps>0$ converge to zero as $t\to \infty$.
The assumption that the exosystem~\eqref{eq:exointro} is infinite-dimensional further leads to situation where the regulation error is not guaranteed to decay at an exponential rate. 
However, it has been observed in~\cite{PauPoh13b,BouBou13} that under suitable assumptions on the plant~\eqref{eq:plantintro} with bounded operators $B$ and $C$ it is possible to achieve rational decay of the regulation error for sufficiently smooth reference and disturbance signals.
In this paper we present new results that establish \keyterm{a priori} decay rates for the regulation error for sufficiently smooth reference and disturbance signal.
The results we present are based on a new method for nonuniform stabilization of the infinite-dimensional internal model in the controller, and on the subsequent analysis of the closed-loop system using recent results on \keyterm{nonuniform stability} of semigroups of operators~\cite{LiuRao05,BatDuy08,BorTom10,BatChi16}.
Beyond obtaining decay rates for the regulation error for the particular controllers we introduce a general methodology for applying the theory of nonuniform stability of semigroups in the study of regular linear control systems.

The following theorem presents a simplified version of the main result regarding the nonuniform decay rates of the regulation error.  The result demonstrates that the rate of decay of the regulation error is dependent on the rate of growth of the norms $\norm{P(i\gw_k)\pinv}$ as $\abs{k}\to \infty$.  
For the detailed assumptions on the system~\eqref{eq:plantintro}, the exosystem~\eqref{eq:exointro} and the controller~\eqref{eq:controllerintro}, see Section~\ref{sec:plantexo}.
A more general version of Theorem~\ref{thm:decayintro} is presented in
Section~\ref{sec:CLnonuniformstab}.
Here we denote by $\norm{x}_{\Dom(A)} = \norm{Ax}+\norm{x}$ the graph norm of a linear operator~$A$.

\begin{thm}
  \label{thm:decayintro}
  Assume $B$ and $B_d$ are bounded, $\sup_{k\in\Z} \norm{R(i\gw_k,A)}<\infty$, and the pair $(\CL,A)$ can be stabilized with bounded output injection. 

  If there exist $M,\ga_0>0$
  such that $\norm{P(i\gw_k)\pinv}\leq M(1+\abs{\gw_k}^{\ga_0})$ for all $k\in\Z$,
  then for any
   $\ga > 2\ga_0+1$
  the controllers in this paper 
  can be construced in such a way that
  \eq{
  \int_t^{t+1} \norm{e(s)}ds\leq M_e^e \left( \frac{\log t}{t} \right)^{\frac{1}{\ga}}
  \left( \norm{x_0}_{\Dom(A)}+ \norm{ z_0}_{\Dom(\mc{G}_1)}+ \norm{v_0}_{\Dom(S)} \right)
  }
  for some $M_e^e>0$ and 
for all $v_0\in \Dom(S)$, $x_0\in \Dom(A)$ and $z_0\in \Dom(\mc{G}_1)$.

If $\norm{P(i\gw_k)\pinv}\leq Me^{\ga_0 \abs{\gw_k}}$ for some $M,\ga_0>0$ and for all $k\in\Z$, 
  then for any
   $\ga > \ga_0$
  the controllers  
  can be construced in such a way that
  \eq{
  \int_t^{t+1} \norm{e(s)}ds\leq \frac{\ga M_e^e}{\log t}
  \left( \norm{x_0}_{\Dom(A)}+ \norm{ z_0}_{\Dom(\mc{G}_1)}+ \norm{v_0}_{\Dom(S)} \right)
  }
  for some $M_e^e>0$ and
for all $v_0\in \Dom(S)$, $x_0\in \Dom(A)$ and $z_0\in \Dom(\mc{G}_1)$.  
\end{thm}

The new controller structure used in this paper was introduced in~\cite{Pau15a} to solve the robust output regulation problem in the situation where robustness is only required with respect to a given class $\Ops_0$ of perturbations. 
This is the case, for example, if a single controller is required to regulate a set $\set{P_1(\gl),\ldots,P_N(\gl)}$ of plants or states of the same plant, or when only some of the parameters of the plant are known to contain uncertainty.
In such a situation it is possible that the internal model in the controller can be replaced with a reduced order internal model containing smaller numbers of copies of some of the frequencies of the exosystem~\cite{PauPoh13a}. In this paper we generalize the controller achieving robustness with respect to a given class $\Ops_0$ of perturbations presented in~\cite{Pau15a,Pau16a} for regular linear systems with infinite-dimensional exosystems. 
Finally, we use the new controller structure to construct a novel controller that achieves output tracking and disturbance rejection without the requirement for robustness with respect to perturbations in the parameters of the plant.

Controller design for robust output regulation for regular linear systems with finite-dimensional exosystems has been studied previously in~\cite{Sai15phd,Pau16a}.
The output tracking and disturbance rejection for regular linear systems without the robustness requirement have been studied in~\cite{NatGil14,XuDub16}. 
In this paper we also extend the characterization of the controllers achieving output regulation via the regulator equations and the internal model principle for regular linear systems in the situation where the operator $B_d$ is unbounded and the controller~\eqref{eq:controllerintro} is a regular linear system. These results generalize the ones in~\cite{PauPoh14a,Sai15phd} where $B_d$ was assumed to be bounded and in~\cite{NatGil14} where feedforward control was studied.
The controllers constructed in this paper are infinite-dimensional due to the 
full order observers
and the internal models of the exosystem~\eqref{eq:exointro}.
Using model reduction methods 
to find
finite-dimensional approximations 
of the presented controllers 
is an important topic for future research.

The paper is structured as follows. In Section~\ref{sec:plantexo} we state the standing assumptions on the plant, the exosystem, the controller, and the closed-loop system. In Section~\ref{sec:RORP} we formulate the main control problems 
and present a generalization of the internal model principle. The robust controller based on the new controller structure is constructed in Section~\ref{sec:ContrOne}. Subsequently, the same structure is used to construct a controller with a reduced order internal model in Section~\ref{sec:ContrROIM} and a controller for output regulation without the robustness requirement in Section~\ref{sec:ContrOneORP}. The observer based controller is constructed in Section~\ref{sec:ContrTwo}. In Section~\ref{sec:CLnonuniformstab} we study the nonuniform decay rates for the state of the closed-loop system and the regulation error. The example on robust output regulation for a two-dimensional heat equation is studied in Section~\ref{sec:heatex}.

\section{The System, The Controller and The Closed-Loop System}
\label{sec:plantexo}

If $X$ and $Y$ are Banach spaces and $A:X\rightarrow Y$ is a linear operator, we denote by $\Dom(A)$, $\ker(A)$ and $\ran(A)$ the domain, kernel and range of $A$, respectively. The space of bounded linear operators from $X$ to $Y$ is denoted by $\Lin(X,Y)$. If \mbox{$A:X\rightarrow X$,} then $\gs(A)$, $\gs_p(A)$ and $\rho(A)$ denote the spectrum, the point spectrum and the \mbox{resolvent} set of $A$, respectively. For $\gl\in\rho(A)$ the resolvent operator is  \mbox{$R(\gl,A)=(\gl -A)^{-1}$}.  The inner product on a Hilbert space is denoted by $\iprod{\cdot}{\cdot}$.
For two sequences $(f_k)_{k\in\Z}\subset X$ and $(g_k)_{k\in\Z}\in \R_+$ we denote $\norm{f_k} = O(g_k)$ if there exist $M_g,N_g>0$ such that $\norm{f_k}\leq M_g g_k$ whenever $\abs{k}\geq N_g$. Similarly, for  $f:I\subset \R\to X$ and $g:\R_+\to \R_+$ we denote $\norm{f(t)}=O(g(t))$ if there exist $M_g,T_g>0$ such that $\norm{f(t)}\leq M_g g(t)$ whenever $\abs{t}\geq T_g$.
We denote $f(t)\lesssim g(t)$ and $f_k\lesssim g_k$ if there exist $M_1,M_2>0$ such that $f(t)\leq M_1 g(t)$ and $f_k\leq M_2g_k$ for all values of the parameters $t$ and $k$.

Throughout the paper we assume that
the plant $(A,B,B_d,C,D)$ in~\eqref{eq:plantintro} is a \keyterm{regular linear system}~\cite{Wei89b,Wei94,StaWei02}
with state $x(t)\in X$, input $u(t)\in U$, output $y(t)\in Y$ and external disturbance $w(t)\in U_d$.
The spaces $X$ and $U_d$ are Banach spaces and $U$ and $Y$ are Hilbert spaces.
The operator $A:\Dom(A)\subset X\rightarrow X$ generates a strongly continuous semigroup $T(t)$ on $X$. For a fixed $\gl_0\in\rho(A)$ we define the scale spaces $X_1 = (\Dom(A), \norm{(\gl_0-A)\cdot})$ and $X_{-1}= \overline{(X,\norm{R(\gl_0,A)\cdot})}$ (the completion of $X$ with respect to the norm $\norm{R(\gl_0,A)\cdot}$)~\cite{TucWei09book},\cite[Sec. II.5]{EngNag00book}.
Also the extensions of the operator $A$ and the semigroup $T(t)$ to the space $X_{-1}$ are denoted by $\Amo: X\subset X_{-1}\rightarrow X_{-1}$ and $\Tmo(t)$, respectively.  
The input and output operators 
$B\in \Lin(U,X_{-1})$, $B_d\in \Lin(U_d,X_{-1})$ and $C\in \Lin(X_1,Y)$ are admissible with respect to $A$ and $D\in \Lin(U,Y)$. 
The admissibility of $B$, $B_d$ and $C$ means that $\int_0^t T(t-s)(Bu(s)+B_dw(s))ds\in X$ for all $t>0$, $u\in \Lp[2](0,t;U)$, and $w\in \Lp[2](0,t;U_d)$, and that for all $t>0$ there exists $\kappa>0$ such that $\int_0^t\norm{CT(s)x}^2ds\leq \kappa^2 \norm{x}^2$ for all $x\in \Dom(A)$.
We define
the $\Lambda$-extension $\CL$ of $C$ as
\ieq{
\CL x = \lim_{\gl\to \infty}\, \gl C R(\gl,A)x
}
with $\Dom(\CL)$ consisting of those $x\in X$ for which the limit exists. If $C\in\Lin(X,Y)$, then $\CL=C$. 
The system $(A,B,B_d,C,D)$ is defined to be \keyterm{regular} if $\ran(R(\gl_0,A)[B,B_d])\subset \Dom(\CL)$ for one/all $\gl_0\in\rho(A)$ and if $\CL R(\gl,A)(Bu+B_dw)\to 0$ as $\gl\to \infty$ with $\gl>0$ for all $u\in U$ and $w\in U_d$.  
For every $x_0\in X$,  $u\in \Lploc[2](0,\infty;U)$, and $w\in \Lploc[2](0,\infty;U_d)$ the system~\eqref{eq:plantintro} has a well-defined mild state and its output $y(t)$ is given by 
\eq{
y(t) = \CL T(t)x_0 +  \CL \int_0^t T(t-s)(Bu(s)+B_dw(s))ds + Du(t)
}
for almost all $t\geq 0$.
The transfer functions $P(\gl)$ and $P_d(\gl)$ from $\hat{u}$ to $\hat{y}$ and from $\hat{w}$ to $\hat{y}$, respectively,  are given by~\cite[Sec. 4]{StaWei02}
\eq{
P(\gl) = \CL R(\gl,A)B + D \qquad \mbox{and} \qquad
P_d(\gl) = \CL R(\gl,A)B_d
}
for all $\gl\in\rho(A)$.
We define 
$X_B=\Dom(A) + \ran(R(\gl_0,\Amo)B)\subset \Dom(\CL)$,
$X_{B_d}=\Dom(A) + \ran(R(\gl_0,\Amo)B_d)\subset \Dom(\CL)$, and 
$\XBBd=\XB+\XBd= \Dom(A)+\ran(R(\gl_0,\Amo)[B, ~B_d])$ 
all of which independent of the choice of $\gl_0\in\rho(A)$.

\begin{ass}
  \label{ass:stabilizability}
  The pairs $(A,B)$ and $(C,A)$ are exponentially stabilizable and detectable, respectively, in the sense that there exist admissible operators $K\in \Lin(X_1,U)$ and $L\in \Lin(Y,X_{-1})$ for which
  \eqn{
  \label{eq:stabassregsys}
  \left( A, \pmat{B, ~ L, ~ B_d}, \pmat{C\\K}\right)
  }
  is a regular linear system and the semigroups generated by $(A+B\KL)\vert_X$ and $(A+L\CL)\vert_X$ are exponentially stable.
\end{ass}

We pose the assumption on the regularity of~\eqref{eq:stabassregsys}
since $K$ and $L$ are allowed to be unbounded.
More details on stabilizability and detectability of regular linear systems can be found in~\cite{Reb93,WeiCur97}.
For techniques for choosing $K$ and $L$, see
for instance~\cite{Reb89a,XuSal96,CurWei06}.

The reference and disturbance signals are generated by an exosystem~\eqref{eq:exointro}
on a separable Hilbert space $W=\lp[2](\Z;\C)$. 
 We denote 
 by
 $\set{\phi_k}_{k\in\Z}$ 
 the canonical orthonormal basis  of $W$.
 The operator $S: \Dom(S)\subset W\to W$ is defined as

  \eq{
  Sv = \sum_{k\in\Z} i\gw_k \iprod{v}{\phi_k}\phi_k, \qquad
  v\in \Dom(S) = \Setm{v\in W}{ \sum_{k\in\Z} \abs{\gw_k}^2 \abs{\iprod{v}{\phi_k}}^2<\infty},
  }
  where the eigenvalues $\set{i\gw_k}_{k\in\Z}\subset i\R$ of $S$ are distinct and have a uniform gap, i.e., $\inf_{k\neq l}\abs{\gw_k-\gw_l}>0$. The operator $S$ generates an isometric group $T_S(t)$ on $W$.  We finally assume that $E\in \Lin(W,X)$ and $F\in \Lin(W,Y)$ are Hilbert--Schmidt operators, i.e., $(E\phi_k)_{k\in\Z}\in \lp[2](U_d)$ and $(F\phi_k)_{k\in\Z}\in \lp[2](Y)$.

All the results presented in this paper also trivially apply to the situation where $W=\C^q$ for some $q\in\N$ and $S=\diag(i\gw_k)_{k=1}^q$ is a diagonal matrix.

\begin{ass}
  \label{ass:PKsurj}
  The  operators $K$ and $L$ in Assumption~\textup{\ref{ass:stabilizability}} can be chosen in such a way that 
  \eq{
  P_K(i\gw_k) &= (\CL + D\KL)R(i\gw_k,A+B\KL)B + D\\
  P_L(i\gw_k) &= \CL R(i\gw_k,\Amo +L\CL)(B+L D) + D
  }
  are surjective for all $k\in\Z$.
\end{ass}

Assumption~\ref{ass:PKsurj} requires that none of the frequencies $\set{i\gw_k}_k\subset \gs(S)$ 
is a \keyterm{transmission zero} of the transfer functions $P_K(\gl)$ or $P_L(\gl)$ of the stabilized systems $(A+B\KL,B,\CL+D\KL,D)$ and $(A+L\CL,B+LD,\CL,D)$, respectively. 
Since transmission zeros are invariant under state feedback and output injection, Assumption~\ref{ass:PKsurj} is 
satisfied whenever $\set{i\gw_k}_k\subset \rho(A)$ and $P(i\gw_k)$ are surjective, and if it satisfied for some $K$ and $L$, then it is satisfied for all stabilizing operators $K$ and $L$. 
In the case
$\set{i\gw_k}_k\subset \rho(A)$ the surjectivity of $P(i\gw_k)$ has been shown to be necessary for robust output regulation,
see~\cite[Cor. V.3]{ByrLau00},\,\cite[Thm. V.2]{NatGil14}, and~\cite[Sec. 5]{PauPoh14a}.  
However, Assumption~\ref{ass:PKsurj} 
can also be used if 
$\set{i\gw_k}_k\cap \gs(A)\neq \varnothing$,
which is the case, e.g., for the heat system in Section~\ref{sec:heatex}.
For examples on verifying Assumption~\ref{ass:PKsurj} and 
locating the zeros of infinite-dimensional
systems, see~\cite[Sec. 5]{ImmPoh06a}, \cite[Sec. 9]{HamPoh10}, and~\cite{CurMor09}. 

In this paper we will solve the robust output regulation problem with a dynamic error feedback controller~\eqref{eq:controllerintro}
on a Banach space $Z$. We assume that $\mc{G}_1:\Dom(\mc{G}_1)\subset Z\rightarrow Z$ generates a semigroup on $Z$, and 
$\mc{G}_2\in\Lin(Y,Z_{-1})$
and $K\in\Lin(Z_1,U)$ are admissible, and $(\mc{G}_1,\mc{G}_2,K)$ is a regular linear system. We denote by $\KL$ the $\Lambda$-extension of $K$ and for $\gl_0\in\rho(\mc{G}_1)$ we define $\ZG = \Dom(\mc{G}_1)+ \ran(R(\gl_0,\Gmo ) \mc{G}_2)\subset \Dom(\KL)$.
Also the extension of $\mc{G}_1$ to the space $Z_{-1}$ is denoted by $\mc{G}_1: Z\subset Z_{-1}\to Z_{-1}$.

The system and the controller can be written together as a closed-loop system on the Banach space $X_{e}=X\times Z$. This composite system with state $x_e(t)=(x(t), z(t))^T$ can be written on $X_{-1}\times Z_{-1}$ as
  \eq{
  \dot{x}_e(t) &= A_ex_e(t)+B_ev(t), \qquad x_e(0)=x_{e0} ,\\
  e(t) &= \CeL x_e(t)+D_ev(t),
  }
where $e(t)= y(t)- \yref(t)$ is the regulation error, $x_{e0}=(x_0, z_0)^T$,  
  \eq{
  A_e=\pmat{\Amo&B\KL\\\mc{G}_2\CL&\Gmo +\mc{G}_2D\KL}, \qquad 
  B_e=\pmat{B_dE\\\mc{G}_2F},
  }
$C_e= \bigl[\CL, ~ D\KL\bigr]$, and
$D_e=F$.  
We also denote 
$B_e^0 =
\left[ {B_d\atop 0} {0\atop \mc{G}_2} \right]
\in \Lin(U_e,X_{-1}\times Z_{-1})$ 
where $U_e = U_d\times Y$
so that $B_e=B_e^0 \left[ E\atop F \right]$.
We choose the domain of $A_e$ as
\eqn{
\label{eq:Aedom}
\Dom(A_e)&=
\Biggl\{\pmat{x\\z}\in \XB\times\ZG ~\Biggm|~ 
\biggl\{ \begin{array}{c}
  \Amo x+B\KL z\in X\\
  \Gmo z+ \mc{G}_2(\CL x+D\KL z)\in Z
\end{array}
\Biggr\} .
}

\begin{thm}
  \label{thm:CLregularity}
  The closed-loop system $\CLops$ is a regular linear system.
  For every $\gl\in\rho(A_e)$ we have
  $\ran(B_e^0)\subset \ran(\gl-A_e^e)$ and
  \ieq{
  R(\gl,\Aemo )B_e^0   = R(\gl,A_e^e)B_e^0 
  }
  where 
  $A_e^e = 
  \left[ {\Amo \atop \mc{G}_2\CL} {B\KL \atop \Gmo +\mc{G}_2D\KL} \right]
  :X_e \to X_{-1}\times Z_{-1}$
  with domain 
 $\Dom(A_e^e)=\XBBd\times \ZG$.
\end{thm}

\begin{proof}
  The results in~\cite[Sec. 7]{Wei94} imply that $(A_e,B_e^0,C_e)$ is regular, since it is part of a system obtained from the regular linear system
\eq{
\left( \pmat{A&0\\0&\mc{G}_1} , \pmat{B&B_d&0\\0&0&\mc{G}_2} ,\pmat{C_\Lambda&0\\0&K_\Lambda} , 
\pmat{D&0&0\\0&0&0} \right) 
}
 with admissible output feedback $\hat{K}= \left[ {0\atop I} {0\atop 0} {I\atop 0}\right]^T$, and thus $\CLops$ is regular as well.
Moreover, the resolvent identities in~\cite[Prop. 6.6]{Wei94} and a straightforward computation shows
that for $\gl\in\rho(A_e)$ with large $\re\gl$ we have 
  \ieq{
  R(\gl,\Aemo )B_e^0 = R(\gl,A_e^e)B_e^0 .
  }
Together with the resolvent identity this implies the last claim of the theorem.
\end{proof}

\subsection{The Class of Perturbations}
\label{sec:pertclass}

In this paper the parameters of the plant are perturbed in such a way that the operators $A$, $B$, $B_d$, $C$, and $D$ are changed to $\tilde{A}: \Dom(\tilde{A})\subset X\rightarrow X$, $\tilde{B}\in \Lin(U,\tilde{X}_{-1})$, $\tilde{B}_d\in \Lin(U,\tilde{X}_{-1})$, $\tilde{C}\in \Lin(\tilde{X}_1,Y)$, and $\tilde{D}\in \Lin(U,Y)$, respectively. Here $\tilde{X}_1$ and $\tilde{X}_{-1}$ are the scale spaces of $X$ related to the operator $\tilde{A}$. We assume that $(\tilde{A},\tilde{B},\tilde{B}_d,\tilde{C},\tilde{D})$ is a regular linear system.
Moreover, the operators $E$ and $F$ are perturbed in such a way that $\tilde{E}\in \Lin(W,X)$ and $\tilde{F}\in \Lin(W,Y)$ are Hilbert--Schmidt.
For $\gl\in \rho(\tilde{A})$ we denote by
$\tilde{P}(\gl) = \tilde{C}_\Lambda R(\gl, \Atmo ) \tilde{B} + \tilde{D}$
and $\tilde{P}_d(\gl) = \tilde{C}_\Lambda R(\gl, \Atmo ) \tilde{B}_d$
the transfer functions of the perturbed plant.

\begin{dfn}
  \label{def:pertclass}
  The class $\Ops$ of considered perturbations consists of operators 
  $(\tilde{A},\tilde{B},\tilde{B}_d,\tilde{C},$ $\tilde{D},\tilde{E},\tilde{F})$ satisfying the above assumptions. 
  In particular, the class $\Ops$ contains the nominal plant, i.e., $\PARsysops\in \Ops$.
\end{dfn}

We denote the operators of the closed-loop system consisting of the perturbed plant and the controller by $\tilde{C}_e=\bigl[\tilde{C}_\Lambda,~\tilde{D}\KL\bigr]$, $\tilde{D}_e = \tilde{F}$ and
\eq{
\tilde{A}_e=\pmat{\Atmo &\tilde{B}\KL\\\mc{G}_2\tilde{C}_\Lambda&\Gmo+\mc{G}_2\tilde{D}\KL},
\qquad \tilde{B}_e = \pmat{\tilde{B}_d\tilde{E}\\\mc{G}_2 \tilde{F}}.
}

\section{The Robust Output Regulation Problem}
\label{sec:RORP}

The main control problem studied in this paper is defined in the following. 

\begin{RORP}
  Choose  $(\mc{G}_1,\mc{G}_2,K)$ in such a way that the following are satisfied:
\begin{itemize}
  \item[\textup{(a)}] The semigroup $T_e(t)$ generated by $A_e$ is strongly stable.
  \item[\textup{(b)}] 
  For all initial states $x_{e0}\in X_e$ and  $v_0\in W$ the regulation error satisfies
\eqn{
\label{eq:errintconv}
\int_t^{t+1} \norm{e(s)}ds\to 0 \qquad \mbox{as} \quad t\to \infty.
}
\item[\textup{(c)}] If
  $(A,B,B_d,C,D,E,F)$ are perturbed to $(\tilde{A},\tilde{B},\tilde{B}_d,\tilde{C},\tilde{D},\tilde{E},\tilde{F})\in \Ops$ 
  in such a way that the perturbed closed-loop system is strongly stable, 
   then for all initial states $x_{e0}\in X_e$ and  
   $v_0\in W$
the regulation error satisfies~\eqref{eq:errintconv}.
\end{itemize}
\end{RORP}

If $B$ and $C$ 
are bounded, then the property~\eqref{eq:errintconv} in the control problem is equivalent to $\norm{e(t)}\to 0$ as $t\to \infty$ (see the proof of Theorem~\ref{thm:ORP} for details).  In Section~\ref{sec:CLnonuniformstab} we will in addition consider rates of convergence of the integrals in~\eqref{eq:errintconv}.

\begin{dfn}
  We call a controller \keyterm{robust with respect to a given perturbation}
$(\tilde{A},\tilde{B},\tilde{B}_d,\tilde{C},$ $\tilde{D},\tilde{E},\tilde{F})\in \Ops$ 
  for which the perturbed closed-loop system is strongly stable if for these perturbed operators the regulation error satisfies $\int_t^{t+1} \norm{e(s)}ds\to 0$ as $t\to \infty$ 
for all  $x_{e0}\in X_e$ and  $v_0\in W$. 
\end{dfn}

\begin{rem}
  \label{rem:ContrPert}
  Although not considered explicitly in the problem,
robust controllers also tolerate perturbations in $\mc{G}_2$ and $K$ as long as the closed-loop stability and any additional conditions are satisfied. This  is advantageous in control design, because it sometimes allows  $K$ and $\mc{G}_2$ to be replaced with sufficiently accurate approximations.
\end{rem}

The control problem without the robustness requirement is referred to as the \keyterm{output regulation problem}.

\begin{ORP}
  Choose the controller $(\mc{G}_1,\mc{G}_2,K)$ in such a way that parts \textup{(a)} and \textup{(b)} of the robust output regulation problem are satisfied.
\end{ORP}

The following theorem characterises the controllers solving the output regulation problem in terms of the solvability of the so-called \keyterm{regulator equations}~\cite{FraWon75a,ByrLau00,HamPoh10}.
The assumption that the Sylvester equation $\Sigma S=\Aemo \Sigma + B_e$ has a solution $\Sigma\in \Lin(W,X_e)$ satisfying $\ran(\Sigma)\subset \Dom(\CeL)$ 
guarantees that for all initial states $v_0\in W$ of the exosystem
the closed-loop state and the regulation error have well-defined behaviour as $t\to \infty$.
The expressions for 
$A_e$, $B_e$, $C_e$, and $D_e$
can be used to formulate the regulator equations using the parameters of the plant~\eqref{eq:plantintro} and the controller~\eqref{eq:controllerintro}.

\begin{thm}
  \label{thm:ORP} 
  Assume the controller\/ $(\mc{G}_1,\mc{G}_2,K)$ stabilizes the closed-loop system strongly in such a way that
  the Sylvester equation $\Sigma S = \Aemo \Sigma + B_e$ on $\Dom(S)$ has a solution $\Sigma\in \Lin(W,X_e)$ 
   satisfying $\ran(\Sigma) \subset \Dom(\CeL)$.
  Then the following are equivalent:
\begin{enumerate}
\item[\textup{(a)}] The controller\/ $(\mc{G}_1,\mc{G}_2,K)$ solves the output regulation problem.
\item[\textup{(b)}] The regulator equations%
  \begin{subequations}%
    \label{eq:regeqns}
    \eqn{
    \label{eq:ORPSyl}
    \Sigma S&=\Aemo \Sigma +B_e  \\
    0&=\CeL \Sigma +D_e \label{eq:ORPRegconstr}
    }
  \end{subequations}%
  on $\Dom(S)$ have a solution  $\Sigma \in \Lin(W,X_e)$  satisfying $\ran(\Sigma)\subset \Dom(\CeL)$.
\end{enumerate}
\end{thm}

\begin{proof}
  Similarly as in the proof of~\cite[Lem. 4.3]{PauPoh14a} equation~\eqref{eq:ORPSyl} implies 
  \eq{
  \int_0^t \Temo(t-s)B_eT_S(s)vds
  =\Sigma T_S(t)v-T_e(t)\Sigma v
  }
  for all $v\in \Dom(S)$, and 
  since the operators on both sides of the equation are in $\Lin(W,X_e)$ and $\Dom(S)$ is dense in $W$, the identity holds for all $v\in W$.
Thus for all $x_{e0}\in X_e$ and $v_0\in W$ the mild state of the closed-loop system has the form
\ieq{
  x_e(t) 
   = T_e(t)(x_{e0}-\Sigma v_0) + \Sigma T_S(t)v_0.
  }
  Since the closed-loop system is regular, we have that $T_e(t)(x_{e0}-\Sigma v_0)\in \Dom(\CeL)$ for almost all $t\geq 0$, and the property $\ran(\Sigma)\subset \Dom(\CeL)$ implies $\Sigma T_S(t)v_0\in \Dom(\CeL)$.
For almost all $t\geq 0$ the regulation error is thus given by%
  \eqn{
  \label{eq:ORPerrformula}
  e(t) &=\CeL x_e(t)+D_ev(t) 
  =\CeL T_{e}(t)(x_{e0}-\Sigma v_0) +(\CeL \Sigma +D_e)T_S(t)v_0.
  }

If (b) is satisfied, then the regulator equations~\eqref{eq:regeqns} have a solution. 
The formula~\eqref{eq:ORPerrformula} and the regulation constraint $\CeL \Sigma +D_e=0$ imply 
$e(t) = \CeL T_{e}(t)(x_{e0}-\Sigma v_0)$ for almost all $t\geq 0$. 
The admissibility of $C_e$ implies that there exists $\kappa>0$ such that
$\int_0^1\norm{\CeL T_e(s)x} ds\leq \kappa \norm{x}$ for all $x\in X_e$. 
Since $T_e(t)$ is strongly stable, we have
\eq{
 \int_t^{t+1} \norm{e(s)}ds
&=   \int_0^1\norm{\CeL T_e(s)T_e(t)(x_{e0}-\Sigma v_0) } ds 
\leq \kappa  \norm{T_e(t)(x_{e0}-\Sigma v_0) }\to 0
} 
as $t\to \infty$.  Thus the controller solves the output regulation problem and (a) holds.

Assume now that (a) is satisfied and the controller solves the output regulation problem.
By assumption,
there exists $\Sigma \in\Lin(\Wf,X_e)$ such that~\eqref{eq:ORPSyl} is satisfied.  
Our aim is to show $\CeL \Sigma + D_e=0$, which will imply that (b) holds. 
The formula~\eqref{eq:ORPerrformula} implies
$ (\CeL \Sigma +D_e)T_S(t)v_0 =e(t) - \CeL T_{e}(t)(x_{e0}-\Sigma v_0)$ and since the closed-loop system is strongly stable and~\eqref{eq:errintconv} holds, we have
\eq{
\int_t^{t+1} \hspace{-1ex} \norm{(\CeL \Sigma +D_e)T_S(s)v_0}ds
\leq \int_t^{t+1} \hspace{-1ex}\norm{\CeL T_{e}(s)(x_{e0}-\Sigma v_0) } ds 
+  \int_t^{t+1}\hspace{-1ex} \norm{e(s)}ds
\to 0 
}
as $t\to \infty$ for all $v_0\in W$ and $x_{e0}\in X_e$.
If $k\in\Z$, $v_0=\phi_k\in W$ and $x_{e0}\in X_e$, then $T_S(s)\phi_k=e^{i\gw_k s}\phi_k$ for all $s\geq 0$ and
\eq{
\norm{(\CeL \Sigma +D_e)\phi_k} =
\int_t^{t+1} \norm{(\CeL \Sigma +D_e)T_S(s)\phi_k}ds \to 0 \quad \mbox{as} \quad t\to\infty,
}
which implies $(\CeL \Sigma +D_e)\phi_k=0$. Since $k\in\Z$ was arbitrary and $\set{\phi_k}_{k\in\Z}$ is a basis of $\Wf$, we have $\CeL \Sigma +D_e=0$, and thus $\Sigma$ is  a solution of~\eqref{eq:regeqns}.
\end{proof}

The proof of Theorem~\ref{thm:ORP} shows that even in a more general situation where $S$ is a generator of a general strongly continuous semigroup $T_S(t)$ on a Banach space $W$, the regulation error has the form~\eqref{eq:ORPerrformula}
whenever 
$\Sigma S=\Aemo \Sigma + B_e$ has a solution $\Sigma\in \Lin(W,X_e)$ satisfying $\ran(\Sigma)\subset \Dom(\CeL)$.
In this situation the regulator equations are sufficient for the solution of the robust output regulation problem. 
The regulator equations are also necessary whenever $T_S(t)$ satisfies a condition of the form
\eq{
 \int_t^{t+1} \norm{QT_S(s)v_0}ds \stackrel{t\to\infty}{\longrightarrow} 0 \qquad \Rightarrow  \qquad Qv_0=0
}
for any $Q\in \Lin(W,Y)$ and $v_0\in W$. This is in particular true for all finite-dimensional exosystems
with $\gs(S)\subset \overline{\C_+}$ as well as for infinite blockdiagonal exosystems.

If $B_d$ and $\mc{G}_2$ are bounded, then $\Sigma$ in~\eqref{eq:ORPSyl} satisfies $\Sigma(\Dom(S))\subset \Dom(A_e)$.
In this situation~\eqref{eq:ORPerrformula} implies
that if~\eqref{eq:regeqns} are satisfied, then
for all $x_{e0}\in \Dom(A_e)$ and $v_0\in \Dom(S)$ the regulation error decays to zero pointwise, i.e., $\norm{e(t)}\to 0$ as $t\to\infty$.

The following lemma presents a sufficient condition for the solvability of~\eqref{eq:ORPSyl}.

\begin{lem}
  \label{lem:Sylsolvabilitycond}
  Assume the closed-loop system is strongly stable,
  $\set{i\gw_k}_{k\in\Z}\subset \rho(A_e)$, and
  \ieq{
  (\RBm)_{k\in\Z}\in \lp[2](X_e).
  }
Then the Sylvester equation $\Sigma S = \Aemo \Sigma + B_e$ on $\Dom(S)$ has a unique solution $\Sigma\in \Lin(W,X_e)$ 
   satisfying $\ran(\Sigma) \subset \Dom(\CeL)$.
\end{lem}

\begin{proof}
  Define  $\Sigma : \Dom(\Sigma)\subset W\to X_e$ by
  \ieq{
  \Sigma v = \sum_{k\in\Z} \iprod{v}{\phi_k} \RBm 
  }
for all $ v\in \Dom(\Sigma)$.
By~\cite[Lem. 6]{HamPoh10} the operator $\Sigma\in \Lin(W,X_{e,-1})$ is the solution of the Sylvester equation $\Sigma S = \Aemo \Sigma + B_e$,
and 
$(\RBm)_{k\in\Z}\in \lp[2](X_e)$
clearly implies $\Sigma \in \Lin(W,X_e)$.
To show   $\ran(\Sigma)\subset \Dom(\CeL )$, we note that the resolvent identity $R(i\gw_k,\Aemo )=R(1+i\gw_k,\Aemo )+R(1+i\gw_k,A_e)R(i\gw_k,\Aemo )$ implies
  $(\iprod{v}{\phi_k}\CeL R(i\gw_k,\Aemo )B_e\phi_k)_{k\in\Z}\in \lp[1](Y)$.
  For all $v\in W$ and $\gl>0$  we have
  \eq{
  \MoveEqLeft[.1] \gl C_e R(\gl,\Aemo ) \Sigma v_0
  = \sum_{k\in\Z} \iprod{v_0}{\phi_k} \gl C_eR(\gl,A_e)R(i\gw_k,\Aemo )B_e\phi_k
  \\
  &
  = \sum_{k\in\Z}  \frac{\gl\iprod{v_0}{\phi_k}}{\gl-i\gw_k} \CeL R(i\gw_k,\Aemo )B_e\phi_k 
  - \CeL R(\gl,A_e)B_e^0 \sum_{k\in\Z}  \frac{\gl\iprod{v_0}{\phi_k}}{\gl-i\gw_k} \left[ E\phi_k \atop F\phi_k \right]
  \\
  &\quad \longrightarrow \sum_{k\in\Z} \iprod{v_0}{\phi_k}  \CeL R(i\gw_k,\Aemo )B_e\phi_k   
  }
  as $\gl\to\infty$ since $(A_e,B_e^0,C_e)$ is regular and since $E$ and $F$ are Hilbert--Schmidt.  Thus $\Sigma v\in \Dom(\CeL)$ by definition.
\end{proof}

\subsection{The Internal Model Principle}

We conclude this section by presenting the internal model principle that characterizes the controllers solving the robust output regulation problem.  We use the following two alternate definitions for an internal model. 

\begin{dfn}
  \label{def:Gconds}
  A controller\/ $(\mc{G}_1,\mc{G}_2,K)$ is said to satisfy the\/ {\em\Gconds} 
  if%
\begin{subequations}%
  \label{eq:Gconds}%
\eqn{
\ran(i\gw_k-\Gmo)\cap\ran(\mc{G}_2)&=\set{0} ~\quad\qquad\qquad \forall k\in\Z , \label{eq:Gconds1} 
\\
\ker(\mc{G}_2)&=\set{0}. \label{eq:Gconds2}
}
\end{subequations}%
\end{dfn}

\begin{dfn}
  \label{def:pcopy}
Assume $\dim Y<\infty$. A controller~$(\mc{G}_1,\mc{G}_2,K)$ is said to\/ {\em incorporate a p-copy internal model} of the exosystem $S$ if for all $k\in\Z$ we 
have
\ieq{
\dim\ker(i\gw_k-\mc{G}_1)\geq\dim\, Y.
}
\end{dfn}

Our first result characterizes the robustness of a controller with respect to \keyterm{individual perturbations} $\PARsysopspert\in\Ops$.
The internal model principle is presented in Theorem~\ref{thm:IMP}.

\begin{thm}
  \label{thm:RORPchar}
  Assume the controller solves the output regulation problem.
Let
$(\tilde{A},\tilde{B},\tilde{B}_d,$ $\tilde{C},\tilde{D},\tilde{E},\tilde{F})\in \Ops$ 
be such that the perturbed closed-loop system is strongly stable, $\set{i\gw_k}_{k\in\Z}\subset\rho(\tilde{A})$, and 
$\tilde{\Sigma} S = \Aetmo  \tilde{\Sigma} + \tilde{B}_e$ has a solution.
  Then the controller is robust with respect to 
$\PARsysopspert$
  if and only if for every $k\in\Z$ the equations%
\begin{subequations}%
    \label{eq:RORPchareqns}
    \eqn{
    \tilde{P}(i\gw_k)\KL z_k &= 
-\tilde{P}_d(i\gw_k)\tilde{E}\phi_k 
    - \tilde{F}\phi_k\\
    (i\gw_k-\mc{G}_1)z_k &= 0
    }
  \end{subequations}%
  have  solutions $z_k\in \Dom(\mc{G}_1)$.
  If~\eqref{eq:RORPchareqns} have a solution $z_k\in \Dom(\mc{G}_1)$, then it is unique. 
\end{thm}

\begin{proof}
  The proof can be completed exactly as the proof of~\cite[Thm. 5.1]{PauPoh14a} (where we replace $E$ by $B_dE$) since
  $R(i\gw_k,\Aetmo )\tilde{B}_e\phi_k= R(i\gw_k,\tilde{A}_e^e)\tilde{B}_e\phi_k$
  by Theorem~\ref{thm:CLregularity}.
\end{proof}

\begin{thm}
  \label{thm:IMP}
  Assume
$\PARcontr$
  stabilizes the closed-loop system strongly in such a way that 
  $\set{i\gw_k}_{k\in\Z}\subset \rho(A_e)$ and 
  \ieq{
  \left( \RBm \right)_{k\in\Z}\in \lp[2](X_e).
  }
If the controller satisfies the \Gconds\ in Definition~\textup{\ref{def:Gconds}}, then it solves the robust output regulation problem. 
The controller is guaranteed to be robust
with respect to all perturbations for which the perturbed closed-loop system is stable, $\set{i\gw_k}_{k\in\Z}\subset \rho(\tilde{A}_e)$ and 
$\Sigma S=\tilde{A}_e \Sigma + \tilde{B}_e$ has a solution $\Sigma \in \Lin(W,X_e)$ satisfying $\ran(\Sigma)\subset \Dom(\CeLt)$.

Moreover, if $\set{i\gw_k}_{k\in\Z}\subset\rho(A)$,
then 
the following hold. 
\begin{itemize}
  \item[\textup{(a)}] The controller solves the robust output regulation problem if and only if it satisfies the \Gconds.
    \item[\textup{(b)}] If $\dim Y=p<\infty$, the controller solves the robust output regulation problem if and only if it incorporates a p-copy internal model of the exosystem, i.e., $\dim \ker(i\gw_k-\mc{G}_1)\geq p$ for all $k\in\Z$.
\end{itemize}
In both cases the controllers are guaranteed to be robust with respect to perturbations for which the perturbed closed-loop system is stable, $\set{i\gw_k}_{k\in\Z}\subset \rho(\tilde{A}) \cap  \rho(\tilde{A}_e)$ and 
$\Sigma S=\tilde{A}_e \Sigma + \tilde{B}_e$ has a solution $\Sigma \in \Lin(W,X_e)$ satisfying $\ran(\Sigma)\subset \Dom(\CeLt)$.

\end{thm}

\begin{proof}
Since 
$\left( \RBm \right)_{k\in\Z}\in \lp[2](X_e)$, 
we have from Lemma~\ref{lem:Sylsolvabilitycond} that $\Sigma S = \Aemo \Sigma + B_e$ has a solution $\Sigma\in \Lin(W,X_e)$ such that $\ran(\Sigma) \subset \Dom(\CeL)$.
Let $\PARsysopspert\in \Ops$
be such that
$\set{i\gw_k}_{k\in\Z}\subset \rho(\tilde{A}_e)$ and 
$\Sigma S=\tilde{A}_e \Sigma + \tilde{B}_e$ has a solution $\Sigma = (\Pi,\Gamma)^T \in \Lin(W,X_e)$ satisfying $\ran(\Sigma)\subset \Dom(\CeLt)$.
If $k\in\Z$, then $\Sigma S \phi_k= \tilde{A}_e \Sigma \phi_k+ \tilde{B}_e\phi_k$ implies $(i\gw_k- \tilde{A}_e)\Sigma \phi_k = \tilde{B}_e\phi_k$. By Lemma~\ref{thm:CLregularity} we have $(\Pi\phi_k,\Gamma\phi_k)^T\in \XBBdt\times \ZG$ and
\eq{
\pmat{(i\gw_k-\tilde{A})\Pi\phi_k - \tilde{B}\KL \Gamma\phi_k \\-\mc{G}_2\CLt\Pi \phi_k + (i\gw_k-\mc{G}_1)\Gamma\phi_k - \mc{G}_2\tilde{D}\KL\Gamma \phi_k} = \pmat{\tilde{B}_d\tilde{E}\phi_k\\ \mc{G}_2\tilde{F}\phi_k}.
}
The second line implies $(i\gw_k-\mc{G}_1)\Gamma\phi_k = \mc{G}_2(\CLt\Pi\phi_k + \tilde{D}\KL\Gamma \phi_k + \tilde{F}\phi_k )$, and the \Gconds\ further imply
$\CeLt \Sigma \phi_k + \tilde{D}_e\phi_k =  \CLt\Pi\phi_k + \tilde{D}\KL\Gamma \phi_k + \tilde{F}\phi_k =0 $.
Since $k\in\Z$ was arbitrary and $\set{\phi_k}_{k\in\Z}$ is a basis of $W$, we have that $\CeLt \Sigma  + \tilde{D}_e=0$.
Since the perturbations in $\Ops$ were arbitrary, we have from
 Theorem~\ref{thm:ORP} that the regulation errors for the nominal and the perturbed systems satisfy $\int_t^{t+1}\norm{e(s)}ds\to 0$ as $t\to\infty$. 

 Under the assumptions $\set{i\gw_k}_{k\in\Z}\subset \rho(A)$ and $\set{i\gw_k}_{k\in\Z}\subset \rho(\tilde{A})$, the rest of the theorem can be proved using  Theorem~\ref{thm:RORPchar} exactly as in~\cite{PauPoh14a}.  
\end{proof}

\begin{lem}
  \label{lem:Gcondsinvariance}
  If
  the operators $(\mc{G}_1,\mc{G}_2)$ satisfy the \Gconds, and if $K: \Dom(K)\subset Z\to Y$ is such that
  $\ker(i\gw_k-\mc{G}_1)\subset \ker(K)$ for all $k\in\Z$, then also 
  $\ran(i\gw_k-\Gmo +\mc{G}_2K)\cap\ran(\mc{G}_2)=\set{0}$ for all $k\in\Z$.
\end{lem}

\begin{proof}
Let $w=(i\gw_k-\Gmo -\mc{G}_2K)z=\mc{G}_2y$ for some $k\in\Z$, $z\in \Dom(K)$ and $y\in Y$. This implies $(i\gw_k-\Gmo )z=\mc{G}_2(y+Kz)\in \ran(i\gw_k-\Gmo )\cap \ran(\mc{G}_2)$, and we thus have $z\in \ker(i\gw_k-\mc{G}_1)$. Due to our assumptions we then also have $Kz=0$ and $w=(i\gw_k-\Gmo )z=\mc{G}_2y$, which finally  imply $w=0$ due to~\eqref{eq:Gconds1}.
\end{proof}

\begin{lem}
  \label{lem:RBeformgeneral}
  Assume the controller $\PARcontr$ satisfies the \Gconds\ and let
$(\tilde{A},\tilde{B},\tilde{B}_d,$ $\tilde{C},\tilde{D},\tilde{E},\tilde{F})\in \Ops$.
  If $k\in\Z$ and $i\gw_k\in \rho(\tilde{A})\cap \rho(\tilde{A}_e)$,
  then
  \eq{
  R(i\gw_k,\Aetmo )B_e\phi_k
  = \pmat{R(i\gw_k,\Atmo )(\tilde{B}Kz_k + \tilde{B}_d\tilde{E}\phi_k)\\z_k}
}
  where $z_k\in Z$ is the unique element such that $z_k\in \ker(i\gw_k-\mc{G}_1)$ and 
\ieq{
\tilde{P}(i\gw_k)Kz_k = -\tilde{P}_d(i\gw_k)\tilde{E}\phi_k - \tilde{F}\phi_k.
}
\end{lem}

\begin{proof}
By Theorem~\ref{thm:CLregularity} we have that  $(x_k,z_k)^T=R(i\gw_k,\Aetmo )\tilde{B}_e\phi_k$ is the  unique element $(x_k,z_k)^T\in \XBBdt\times \ZG$ satisfying 
\eq{
\left\{
\begin{array}{l}
  (i\gw_k-\Atmo )x_k  = \tilde{B}\KL z_k + \tilde{B}_d\tilde{E}\phi_k\\
  (i\gw_k-\Gmo )z_k = \mc{G}_2 (\CLt x_k + D\KL z_k + \tilde{F}\phi_k).
\end{array}
\right.
}
Thus $x_k = R(i\gw_k,\Atmo )(\tilde{B}\KL z_k + \tilde{B}_d\tilde{E}\phi_k)$ and the
\Gconds~\eqref{eq:Gconds} imply $z_k\in \ker(i\gw_k-\mc{G}_1)$ and 
\ieq{
0=\CLt x_k + \tilde{D}\KL z_k + \tilde{F}\phi_k 
= \tilde{P}(i\gw_k) \KL z_k +  \tilde{P}_d(i\gw_k) \tilde{E}\phi_k + \tilde{F}\phi_k.
}
\end{proof}

\section{The New Controller Structure}
\label{sec:ContrOne}

In this section we construct a robust controller using the general internal model based structure introduced  in~\cite{Pau16a,Pau15a}.
The construction of the controller 
is completed in steps, and the required properties 
of the parts of the controller are verified in the proof of Theorem~\ref{thm:ContrOneMain}.

\medskip

\noindent\textbf{Step $\bm{1}^\circ$:} 
The state space of the controller is chosen as $Z=\ZI\times X$ and we choose the structures of the operators $(\mc{G}_1,\mc{G}_2,K)$ as
  \eq{
  \mc{G}_{1}=\pmat{G_1&G_2(\CL +D\KtL)\\0&\Amo +B\KtL+L(\CL +D\KtL)}, 
\qquad   \mc{G}_{2}=\pmat{G_2\\L},
  }
  and $K = (\KoL, \; -\KtL)$. 
Due to Assumption~\ref{ass:stabilizability} concerning the stabilizability of $(A,B)$ and the detectability of $(C,A)$, we can choose
$K_2\in\Lin(X_1,U)$ and $L_1\in\Lin(Y,X_{-1})$ in such a way that 
$(\Amo +B\KtL)\vert_X$ and $(A+L_1\CL)\vert_X $ generate exponentially stable semigroups and  $\left( A,[B,~ L_1, ~B_d], \pmatsmall{C\\K_2},D \right)$ is regular. 
For 
$\gl\in \overline{\C_+}$ 
we define
 \eq{
 P_L(\gl) = \CL R(\gl,\Amo +L_1\CL)(B+L_1 D) + D.
 }
 We have from~\cite[Sec. 7]{Wei94} that $(A+L_1\CL,B+L_1D,\CL,D)$ is a regular linear system, and thus $\sup_{\gw\in \R}\norm{P_L(i\gw)}<\infty$.
 
 \medskip
\noindent\textbf{Step $\bm{2}^\circ$:}
  The operator $G_1$ is the internal model of the exosystem~\eqref{eq:exointro}, and it is defined by choosing
  $\ZI = \lp[2](Y)$, and
\eq{
G_1 = \diag \bigl( i\gw_k I_Y\bigr)_{k\in\Z}, \qquad \Dom(G_1) = \setm{(z_k)_{k\in\Z}\in \ZI}{(\gw_kz_k)_{k\in\Z}\in \lp[2](Y)}.
}
The operator
$K_1 = \bigl(\ldots,K_{1,-1},K_{10},K_{11},\ldots)\in \Lin( \ZI,U)$ is assumed to be Hilbert--Schmidt (i.e., $(\norm{K_{1k}})_{k\in\Z}\in \lp[2](\C)$, which is in particular always true if $\dim U<\infty$) and 
$G_2 =(G_{2k})_{k\in\Z}\in  \Lin(Y, \ZI)$.
We choose the components $K_{1k}\in \Lin(Y,U)$ of $K_1$ in such a way that $P_L(i\gw_k)K_{1k}\in \Lin(Y)$ are boundedly invertible. 
This is possible since $P_L(i\gw_k)$ are surjective by Assumption~\ref{ass:PKsurj}.
For example, we can choose  $K_{1k}=\gg_k \frac{P_L(i\gw_k)\pinv}{\norm{P_L(i\gw_k)\pinv}}$ for all $k\in \Z$ and for a sequence $(\gg_k)_{k\in\Z}\subset \lp[2](\C)$ satisfying $\gg_k\neq 0$ for all $k\in\Z$.
For more concrete choices of $K_{1k}$, see Corollary~\ref{cor:ContrOnePolExp}.

\medskip

If $i\gw_k\in \rho(A)$ for some $k$, then 
$P_L(i\gw_k) = (I-\CL R(i\gw_k,A)L_1)\inv P(i\gw_k)$
 implies that $P_L(i\gw_k)K_{1k}$ is boundedly invertible if and only if $P(i\gw_k)K_{1k}$ is boundedly invertible.

 \medskip

\noindent\textbf{Step $\bm{3}^\circ$:} 
We define $H\in\Lin(\ZI,X)$ by
\eq{
Hz 
= \sum_{k\in\Z} R(i\gw_k,\Amo +L_1\CL)(B+L_1D)K_{1k}z_k, 
\qquad \mbox{for} ~~ z=(z_k)_{k\in\Z}\in \ZI.
}

\medskip

\noindent\textbf{Step $\bm{4}^\circ$:}
We choose  $G_2\in \Lin(Y,\ZI)$ as
\eq{
G_2 = (G_{2k})_{k\in\Z} = (-(P_L(i\gw_k)K_{1k})^\ast)\kZ.
}
Finally, we define $L=L_1 + HG_2\in \Lin(Y,X_{-1})$ and 
choose the domain of  $\mc{G}_1$ to be
\eq{
\Dom(\mc{G}_1) = \Dom(G_1)\times \Dom( (\Amo +(B+LD)\KtL+L\CL)\vert_X).
}

The following theorem presents conditions for the solvability of the robust output regulation problem. It should be noted that $(\norm{P_L(i\gw_k)K_{1k}})_{k\in\Z}\in \lp[2](\C)$ implies that $\norm{(P_L(i\gw_k)K_{1k})\inv}\to \infty$ as $\abs{k}\to \infty$, and thus the condition~\eqref{eq:ContrOneEFcond} requires that $\norm{E\phi_k}$ and $\norm{F\phi_k}$ decay sufficiently fast as $\abs{k}\to \infty$.

\begin{thm}
  \label{thm:ContrOneMain}
If $E\in \Lin(W,\Uw)$ and $F\in \Lin(W,Y)$ satisfy%
\begin{subequations}%
  \label{eq:ContrOneEFcond}
  \eqn{
  &\left( \norm{\CL R(i\gw_k,\Amo +L_1\CL)B_d} \norm{(P_L(i\gw_k)K_{1k})\inv} \norm{E\phi_k}\right)_{k\in\Z}\in \lp[2](\C)\\
  &\left( \norm{(P_L(i\gw_k)K_{1k})\inv}\norm{F\phi_k}) \right)_{k\in\Z}\in \lp[2](\C),
  }
\end{subequations}%
then the controller solves the robust output regulation problem.

The controller is then guaranteed to be robust with respect to all perturbations in $\Ops$ for which the strong closed-loop stability is preserved, $\set{i\gw_k}_{k\in\Z}\subset\rho(\tilde{A}_e)$, $\set{i\gw_k}_{\abs{k}\geq N}\subset\rho(\tilde{A})$ for some $N\in\N$, $\tilde{P}(i\gw_k)K_{1k}$ are invertible
whenever  $\abs{k}\geq N$, and for which%
\begin{subequations}%
  \label{eq:ContrOneEFcondpert}
  \eqn{
  &\left( \norm{R(i\gw_k,\Atmo )\left( \tilde{B}_d\tilde{E}\phi_k- \tilde{B}K_{1k}(\tilde{P}(i\gw_k)K_{1k})\inv \tilde{y}_k \right)} \right)_{\abs{k}\geq N} \in \lp[2](\C)\\
  &\left(\norm{(\tilde{P}(i\gw_k)K_{1k})\inv \tilde{y}_k}\right)_{\abs{k}\geq N} \in \lp[2](\C) 
  } 
\end{subequations}%
where
 $\tilde{y}_k = \tilde{P}_d(i\gw_k)\tilde{E}\phi_k + \tilde{F}\phi_k$.
\end{thm}

In the proof of Theorem~\ref{thm:ContrOneMain} we will see that if $\set{i\gw_k}_{k\in\Z}\subset \rho(A)$, then 
the conditions~\eqref{eq:ContrOneEFcond} can alternatively be replaced with conditions~\eqref{eq:ContrOneEFcondpert} for the nominal plant. However, the condition~\eqref{eq:ContrOneEFcond} has the advantage that the condition only involves the resolvent and the transfer function of the stabilized plant.
This is an advantage if $i\gw_k\notin\rho(A)$ for some $k$, as is the case in the example in Section~\ref{sec:heatex}.
In the situation where $\sup_{\abs{k}\geq N} \norm{R(i\gw_k,A)}<\infty$ for some $N\in\N$, also the norms 
$ \norm{R(i\gw_k,\Amo )B}$ and $\norm{R(i\gw_k,\Amo )B_d}$ are uniformly bounded with respect to $k\in\Z$ with $\abs{k}\geq N$, and the condition~\eqref{eq:ContrOneEFcondpert} simplifies to the following form.

\begin{cor}
  \label{cor:ContrOneAltEF}
If there exists $N\in\N$ such that $\set{i\gw_k}_{\abs{k}\geq N}\subset\rho(A)$ and
$\vspace{-1.6ex}\displaystyle\sup_{\abs{k}\geq N} \norm{R(i\gw_k,\Amo )}<\infty$,
  then the conclusions of Theorem~\textup{\ref{thm:ContrOneMain}} hold if 
\eq{
\left( \norm{(P(i\gw_k)K_{1k})\inv}(\norm{P_d(i\gw_k) E\phi_k}  + \norm{F\phi_k} ) \right)_{\abs{k}\geq N} \in \lp[2](\C)
    }
    and the controller is guaranteed to be robust with respect to perturbations in $\Ops$ for which the strong closed-loop stability is preserved, 
    $\set{i\gw_k}_{k\in\Z}\subset\rho(\tilde{A}_e)$,$\set{i\gw_k}_{\abs{k}\geq N}\subset \rho(\tilde{A})$,
    $\sup_{\abs{k}\geq N} \norm{R(i\gw_k,\Atmo )}<\infty$,
and
    \eq{
    \left( \norm{(\tilde{P}(i\gw_k)K_{1k})\inv}(\norm{\tilde{P}_d(i\gw_k) \tilde{E}\phi_k}  + \norm{\tilde{F}\phi_k} ) \right)_{\abs{k}\geq N} \in \lp[2](\C).
    }
\end{cor}

It should also be noted that if $\norm{\CL R(i\gw_k ,A)L_1}$ is uniformly bounded for large $\abs{k}$, then
the asymptotic rate of $\norm{P_L(i\gw_k)\pinv}$ is at most equal to the rate of  $\norm{P(i\gw_k)\pinv}$.
Due to our assumptions the norms $\norm{\CL R(i\gw_k,A_L)B_d}$ are uniformly bounded with respect to $k\in\Z$, and it is possible that $\norm{\CL R(i\gw_k,A_L)B_d}\to 0$ as $k\to \pm \infty$. Thus the summability condition for  $(\norm{E\phi_k})_{k\in\Z}$ in~\eqref{eq:ContrOneEFcond} is in general weaker or equivalent compared to the summability condition for the sequence  $(\norm{F\phi_k})_{k\in\Z}$.

The following corollary presents specific choices of $K_1$ 
in the cases where
$\norm{P_L(i\gw_k)\pinv}$ are 
either
polynomially or exponentially bounded.

\begin{cor}
  \label{cor:ContrOnePolExp}
  The following hold. 
  \begin{itemize}
    \item[\textup{(a)}] If 
      $\norm{P_L(i\gw_k)\pinv}=\Omi(\abs{\gw_k}^\ga) $ for some $\ga>0$ and $(\gg_k)_{k\in\Z}\in \lp[2](\C)$, the choice 
      \eqn{
      \label{eq:ContrOneK1choice}
      K_{1k} = \gg_k\frac{P_L(i\gw_k)\pinv}{\norm{P_L(i\gw_k)\pinv}}  
      }
      solves the robust output regulation problem for operators $E$ and $F$ that satisfy
      \eq{
      \left( \frac{\abs{\gw_k}^\ga }{\abs{\gg_k}} (\norm{\CL R(i\gw_k,A+L_1\CL)B_d}\norm{E\phi_k}+\norm{F\phi_k}) \right)_{k\in\Z}\in \lp[2](\C).
      \hspace{-5ex}%
      }
      If we in particular choose $\gg_k = \abs{\gw_k}^{-\gb}$  for some $\gb>1/2$ whenever $\gw_k\neq 0$, then $\abs{\gw_k}^\ga /\abs{\gg_k} = \abs{\gw_k}^{\ga+\gb}$ whenever $\gw_k\neq 0$.
      
    \item[\textup{(b)}] If
      $\norm{P_L(i\gw_k)\pinv}=\Omi(e^{\ga\abs{\gw_k}}) $ for some $\ga>0$ and $(\gg_k)_{k\in\Z}\in\lp[2](\C)$, the choice~\eqref{eq:ContrOneK1choice}
solves the robust output regulation problem for operators $E$ and $F$ that satisfy
      \eq{
      \left( \frac{e^{\ga\abs{\gw_k}} }{\abs{\gg_k}} (\norm{E\phi_k}+\norm{F\phi_k}) \right)_{k\in\Z}\in \lp[2](\C).
      }
  \end{itemize}
\end{cor}

We begin
the proof of Theorem~\ref{thm:ContrOneMain} 
by considering the stability properties of the semigroup generated by $G_1-G_2G_2^\ast$. 
If $\dim Y_k<\infty$ the stability
follows from~\cite{Ben78a}.

\begin{lem}
  \label{lem:G1fbstab}
  Assume $U$ and $Y_k$ for $k\in\Z$ are Hilbert spaces.
  Consider
  $\ZI = \setm{(z_k)_{k\in\Z}\in\mbox{\scalebox{1.4}{$\otimes$}}_{k\in\Z}Y_k}{\sum_{k\in\Z}\norm{z_k}_{Y_k}^2<\infty}$ with inner product $\iprod{z}{v}=\sum_{k\in\Z}\iprod{z_k}{v_k}_{Y_k}$ for $z=(z_k)_k$ and $v=(v_k)_k$. 
Assume $\set{i\gw_k}_{k\in\Z}$ has no finite accumulation points, and let $G_1 = \diag(i\gw_k I_{Y_k})\kZ$ on $\ZI$ with domain $\Dom(G_1)=\setm{(z_k)\kZ\in \ZI}{(\gw_kz_k)_k\in\ZI}$ and $G_2 = (G_{2k})_{k\in\Z} \in \Lin(U,\ZI)$.  If the components $G_{2k}\in \Lin(U,Y_k)$ of $G_2$ are surjective and $G_{2k}^\ast$ have closed ranges for all $k\in\Z$, then the semigroup generated by $G_1-G_2G_2^\ast$ is strongly stable and $i\R\subset \rho( G_1-G_2G_2^\ast)$.

  Moreover, 
if $G_{2k}$ are boundedly invertible for all $k\in \Z$, 
then
  \ieq{
  \norm{R(i\gw_k,G_1-G_2G_2^\ast)G_2}=\norm{G_{2k}\inv}
  } 
  for all for $k\in\Z$.
\end{lem}

\begin{proof}
Since $G_1$  generates a contraction semigroup,
the same is true for 
$G_1-G_2 G_2^\ast$~\cite[Cor. III.2.9]{EngNag00book}, and 
$\gs(G_1-G_2G_2^\ast)\subset \overline{\C_-}$.  
The strong stability of the semigroup follows from the Arendt--Batty--Lyubich--V\~{u} Theorem~\cite{AreBat88,LyuVu88} once we show $i\R\subset \rho(G_1-G_2G_2^\ast)$.

  Let $i\gw\in i\R$ be such that $\gw\neq \gw_k$ for all $k\in\Z$. We have $i\gw\in \rho(G_1)$, and if $I+G_2^\ast R(i\gw,G_1)G_2 $ is boundedly invertible, then the Woodbury formula%
    \eq{
     R(i\gw,G_1-G_2 G_2^\ast)
    = R(i\gw,G_1) 
    [I - G_2(I+G_2^\ast R(i\gw,G_1)G_2)\inv G_2^\ast R(i\gw,G_1)]
    }
  implies that $i\gw-G_1+ G_2 G_2^\ast$
  has a bounded inverse. Since $ G_2^\ast R(i\gw, G_1)G_2$ is bounded and skew-adjoint, we have $1\in \rho(-G_2^\ast R(i\gw,G_1)G_2)$ and $i\gw\in \rho( G_1-G_2 G_2^\ast)$.

Assume now $i\gw=i\gw_n$ for some $n\in\Z$.
We will first show that
there exists $c>0$ such that $\norm{(i\gw_n-G_1 +  \BI\BI^\ast )z }\geq c \norm{z}$ for all $z\in \Dom( G_1-G_2G_2^\ast)$. If this is not true, there exists a sequence $(z_k)_{k\in\N}\subset \Dom( G_1-G_2G_2^\ast)$ such that $\norm{z_k}=1$ for all $k\in \N$ and $\norm{(i\gw_n-G_1 +  \BI\BI^\ast)z_k }\to 0$ as $k\to \infty$.
Since $i\gw_n-G_1$ is skew-adjoint, we have
\eq{
\MoveEqLeft \norm{(i\gw_n-G_1+G_2G_2^\ast )z_k}
\geq \abs{\re\iprod{(i\gw_n-G_1+G_2G_2^\ast )z_k}{z_k})}
= \norm{G_2^\ast z_k}^2,
}
and thus $\norm{G_2^\ast z_k} \to 0$ as $k\to \infty$. 
For every $k\in\N$ denote
$z_k=z_k^1+z_k^2$ where $z_k^1\in \ran(i\gw_n-G_1)$, $z_k^2\in\ker(i\gw_n-G_1)$, and $1=\norm{z_k}^2=\norm{z_k^1}^2+\norm{z_k^2}^2$.
There exists $c_1>0$ such that
$\norm{(i\gw_n-G_1)z_k^1}\geq c_1 \norm{z_k^1}$ for all $k\in\N$.
Thus
\eq{
 c_1\norm{z_k^1}
 &\leq
\norm{(i\gw_n-G_1)z_k}
\leq \norm{(i\gw_n-G_1+G_2G_2^\ast)z_k}+ \norm{G_2}\norm{G_2^\ast z_k}
\to 0 
}
as $k\to\infty$.
Moreover, $\norm{G_2^\ast z_k^2}\geq \norm{(G_{2n}^\ast)\pinv}\inv \norm{z_k^2}$, and 
\eq{
\norm{(G_{2n}^\ast)\pinv}\inv \norm{z_k^2}\leq 
\norm{G_2^\ast z_k^2}\leq  \norm{G_2^\ast z_k}+\norm{G_2^\ast z_k^1}\to 0
}
as $k\to\infty$.
We have now shown that
 $z_k^1\to 0$ and $z_k^2\to 0$, but this contradicts the assumption that $ \norm{z_k}^2 =1$ for all $k\in\N$, and 
 thus shows that $i\gw_n-G_1 +  \BI\BI^\ast$ is lower bounded. 
In particular we now have $i\gw_n\notin \gs_p(G_1+\BI\BI^\ast)$, and that the range of $G_1+\BI\BI^\ast$ is closed. Finally, the Mean Ergodic Theorem~\cite[Sec. 4.3]{AreBat01book} implies that the range of $G_1+\BI\BI^\ast$ is dense, and we thus have $i\gw_n\in\rho(G_1+\BI\BI^\ast)$.

The structure of $G_1$ and the assumption that the components $G_{2k}$ of $G_2$ are boundedly invertible imply that $\ker(G_2)=\set{0}$ and $\ran(i\gw_k-G_1)\oplus \ran(G_2)=\ZI$ for all $k\in\Z$.  Let $k\in\Z$ and $u\in U$, and denote $z=R(i\gw_k,G_1-G_2G_2^\ast)G_2u$. Then $\ran(i\gw_k-G_1)\cap \ran(G_2)=\set{0}$ and $\ker(G_2)=\set{0}$ imply
\eq{
(i\gw_k-G_1)z=G_2(u-G_2^\ast z)
\quad \Leftrightarrow \quad
\left\{
\begin{array}{l}
  (i\gw_k-G_1)z=0\\
  u=G_2^\ast z.
\end{array}
\right.
}
Thus $z=(z_l)_{l\in\Z}\in \ker(i\gw_k-G_1)$ is such that $z_l=0$ for all $l\neq k$, and $z_k=G_{2k}^{-\ast}u$
which further implies $\norm{z}_{\ZI}=\norm{G_{2k}^{-\ast}u}_{Y}$.
Since $u\in U$ was arbitrary, this implies
$\norm{R(i\gw_k,G_1-G_2G_2^\ast)G_2} = \norm{G_{2k}\inv}$.
\end{proof}

\noindent\textit{Proof of Theorem~\textup{\ref{thm:ContrOneMain}}.} 
  We can complete the proof by showing that the controller $(\mc{G}_1,\mc{G}_2,K)$ constructed in Section~\ref{sec:ContrOne} has the following properties:
  \begin{itemize} 
     \item[\textup{(i)}] 
       The controller $(\mc{G}_1,\mc{G}_2,K)$ is regular and it satisfies the \Gconds~\eqref{eq:Gconds}.
\item[\textup{(ii)}] The operator 
$G_1+G_2(\CL H+DK_1)$ generates a strongly stable semigroup.
    \item[\textup{(iii)}] The semigroup generated by $A_e$ is strongly stable and $\gs(A_e)\subset \C_-$.
    \item[\textup{(iv)}] There exists $M_e\geq 0$ such that 
      \eq{
\MoveEqLeft      \norm{\RBm}\leq M_e 
\max \bigl\{\norm{E\phi_k},\norm{F\phi_k},
\\
&\hspace{1.2cm} \norm{(P_L(i\gw_k)K_{1k})\inv}(\norm{\CL R(i\gw_k,A+L_1\CL)B_d}\norm{E\phi_k}+\norm{F\phi_k})
\bigr\}.
      }
    \item[\textup{(v)}] $(\norm{R(i\gw_k,\Aetmo )\tilde{B}_e\phi_k})\in \lp[2](\C)$ if and only if~\eqref{eq:ContrOneEFcondpert} are satisfied.
\end{itemize}

The results in~\cite[Sec. 7]{Wei94} show that the controller $\PARcontr$ is a regular linear system on $Z=\ZI\times X$ since it is a part of the system obtained from 
\eq{
\left( \pmat{G_1&0\\0&A}, \pmat{G_2&0\\L&B}, \pmat{\KoL&-\KtL\\0&\CL\\0&\KtL}, \pmat{0&0\\0&D\\0&0} \right),
}
with admissible output feedback
$\hat{K} = \left[ {0\atop 0} {I\atop 0} {0\atop I} \right]$.
It is also straightforward to verify that for all $\gl\in\rho( \mc{G}_1)$
we have
\ieq{
R(\gl,\Gmo )\mc{G}_2 = 
R(\gl,\mc{G}_1^e)\mc{G}_2 
}
where $\mc{G}_1^e: \Dom(\mc{G}_1^e)\subset Z\to \ZI\times X_{-1}$ is the operator $\mc{G}_1$ with domain $\Dom(\mc{G}_1^e) = \Dom(G_1)\times \XBL$.

We will now show that $(\mc{G}_1,\mc{G}_2,K)$ satisfies the \Gconds.  The operator $G_2$ is of the form $G_2=(G_{2k})_{k\in\Z}$, and its components $G_{2k} = -(P_L(i\gw_k)K_{1k})^\ast$ are boundedly invertible for all $k\in \Z$.
This implies $\ker(G_2)=\set{0}$, and also further shows that $\ker(\mc{G}_2)=\set{0}$.
Let $n\in\Z$ be arbitrary and assume $(w,v)^T=(i\gw_n-\Gmo )(z,x)^T=\mc{G}_2y\in \ran(i\gw_n-\Gmo )\cap\ran(G_2)$ for some $(z,x)^T\in Z\times X$ and $y\in Y$. The 
property $R(\gl,\Gmo )\mc{G}_2 = 
R(\gl,\mc{G}_1^e)\mc{G}_2$
implies that $(z,x)\in \Dom(G_1)\times \XBL$ and
\eq{
\pmat{w\\v}
=\pmat{(i\gw_n-G_1)z -G_2C_Kx\\ (i\gw_n- \Amo -B\KtL -LC_K )x}  = \pmat{G_2y\\Ly}
}
where $C_K = \CL +D\KtL$. The first line implies $(i\gw_n-G_1)z=G_2(C_Kx+y)$, which means that in particular
$G_{2n}(C_Kx+y)=(i\gw_n-i\gw_n)z_n=0$.  
Since $G_{2n}$ is injective, we must have
$C_Kx+y=0$.
Using this, the second line above implies
$(i\gw_n-\Amo -B\KtL)x = L(C_Kx+y)=0$, and  since $i\gw_n\in \rho( (\Amo +B\KtL)\vert_X)$, we must have $x=0$. Using this we get $y=-C_Kx=0$, and thus also $(w,v)^T=\mc{G}_2y=0$. This concludes that $\ran(i\gw_n-\Gmo )\cap \ran(\mc{G}_2)=\set{0}$. Since $n\in\Z$ was arbitrary, 
the \Gconds\ are satisfied.

The system $(A+L_1\CL,B+L_1D,\CL,D)$ is exponentially stable and regular~\cite[Sec. 7]{Wei94} and
$R(\gl,\Amo +L_1\CL)\ran(B+L_1D)\subset X_B$ by~\cite[Prop. 6.6]{Wei94}.
There exists $M\geq 0$ such that $ \norm{P_L(i\gw_k)}\leq M$ and 
$\norm{R(i\gw_k,\Amo +L_1\CL)(B+L_1D)}\leq M$ for all $k\in\Z$,
and thus $G_2\in \Lin(Y,\ZI)$ and $H\in \Lin(\ZI,X)$.

Denote $A_L=\Amo +L_1\CL$ and $B_L=B+L_1D$.
If $z\in Z$ and $\gl>0$
then analogously as in the proof of Lemma~\ref{lem:Sylsolvabilitycond} we have
\ieq{
\gl C R(\gl,A_L)Hz 
\to
 \sum_{k\in\Z}  \CL R(i\gw_k,A_L)B_LK_{1k}z_k
}
as $\gl\to \infty$ since $(A_L,B_L,C)$ is regular and since $(K_{1k}z_k)_{k\in\Z}\in \lp[2](U)$. Thus $\ran(H)\subset \Dom(\CL)$, and similarly $\ran(H)\subset \Dom(\KtL)$.
These properties imply that we can define $C_1=\CL H+DK_1\in \Lin(\ZI,Y)$ and 
\eq{
\CL Hz + DK_1 z  
&=\sum_{k\in\Z}  (\CL R(i\gw_k,A_L)B_L + D)K_{1k}z_k 
=\sum_{k\in\Z}  P_L(i\gw_k)K_{1k}z_k .
}
Thus $G_2 = -C_1^\ast$.
The fact that $K_{1k}$ were chosen so that $P_L(i\gw_k)K_{1k}$ are boundedly invertible imply that the components $G_{2k}$ of $G_2$ are boundedly invertible for all $k\in \Z$.
  We thus have from Lemma~\ref{lem:G1fbstab}
that the semigroup generated by $G_1+G_2C_1=G_1-G_2 G_2^\ast$ is strongly stable, $i\R\subset \rho(G_1+G_2C_1)$ and $\norm{R(i\gw_k,G_1+G_2C_1)G_2}=\norm{(P_L(i\gw_k)K_{1k})\inv}$ for all $k\in\Z$.

We will now show that the closed-loop system is strongly stable and $i\R\subset \rho(A_e)$. 
With the chosen controller $(\mc{G}_1,\mc{G}_2,K)$ 
the operator $A_e$ becomes
\eq{
A_{e}
=\pmat{\Amo &BK_1 &-B\KtL \\G_2\CL &G_1+G_2DK_1&G_2\CL \\L\CL &LDK_1 &\Amo +BK_2+L\CL }
}
with domain $\Dom(A_e)$ equal to
\eq{
\Dom(A_e) = 
\Biggl\{ 
\pmat{x\\z_1\\x_1} &\in
\XB\times \Dom(G_1)\times \XBL
~\Biggm|~\\
&\biggl\{
\begin{array}{l}
  \Amo x+BK_1z_1-B\KtL x_1\in X\\
  (\Amo +B\KtL + L\CL)x_1 + L\CL x + LDK_1 z_1\in X
\end{array}
\Biggr\}.
} 
If we choose a similarity transform $Q_{e}\in \Lin(X\times \ZI\times X)$ 
\eq{
Q_{e}
=\mbox{\small$\displaystyle\pmat{I&0&0\\0&I&0\\-I&H&-I}$}
=  Q_{e}\inv ,
}
we can define $\hat{A}_e = Q_{e}A_{e}Q_{e}\inv$ on $X\times \ZI\times X$. 
The operator $H$ is the unique solution of the Sylvester equation $HG_1 =(\Amo +L_1\CL )H+(B+L_1D)K_1$.
Analogously as in the proof of~\cite[Thm. 12]{Pau16a} we can see that
\eq{
\Dom(\hat{A}_e) 
= \Biggl\{ \pmat{x\\z_1\\x_1} &\in  \XB\times \Dom(G_1)\times \XL
~\Biggm|~\\
&\biggl\{
\begin{array}{l}
  (\Amo +B\KtL)x+B(K_1-\KtL H)z_1+B\KtL x_1\in X\\
(\Amo + L_1\CL)x_1   \in X
\end{array}
\Biggr\}
} 
where we have denoted $\XL = \Dom(A)+\ran(R(\gl_0,\Amo )L) =\Dom(A)+\ran(R(\gl_0,\Amo )L_1)$.
For $x_e=(x,z_1,x_1)^T\in  \Dom(\hat{A}_e)$ 
a direct computation using $L=L_1+HG_2$, $C_1=\CL H+DK_1$, and $HG_1z_1 =(\Amo +L_1\CL )Hz_1+(B+L_1D)K_1z_1$ yields
\eq{
\hat{A}_{e}x_e
&=\pmat{(\Amo  + B\KtL )x +B(K_1 -\KtL H)z_1 +B\KtL x_1\\(G_1+G_2(\CL H+DK_1))z_1 - G_2\CL  x_1\\ (\Amo  +L_1\CL)  x_1}\\
&=\pmat{\Amo  + B\KtL  & B(K_1 -\KtL H) &B\KtL  \\0& G_1+G_2C_1 &- G_2\CL  \\ 0&0& \Amo  +L_1\CL} \pmat{x\\z_1\\   x_1}.
}
Since $G_1+G_2C_1$ is strongly stable and $i\R\subset \rho(G_1+G_2C_1)$, and since
$(\Amo +B\KtL )\vert_X$ and $(A+L_1\CL)\vert_X $  generate exponentially stable semigroups and $\KtL H\in \Lin(\ZI,U)$, 
the semigroup generated by $\hat{A}_e$ is strongly stable and $i\R\subset \rho(\hat{A}_e)$.
Due to similarity, the same is true for $A_e$, and thus the closed-loop system is strongly stable.

To guarantee the solvability of the Sylvester equation $\Sigma S = \Aemo \Sigma + B_e$ we need to estimate $\norm{R(i\gw_k,\Aemo )B_e\phi_k}$ for $k\in\Z$. By Theorem~\ref{thm:CLregularity} we have $R(i\gw_k,\Aemo )B_e\phi_k=R(i\gw_k,A_e^e)B_e\phi_k$ where $A_e^e$ is the operator $A_e$ with the domain $\Dom(A_e^e) = \XBBd\times \ZG = \XBBd \times \Dom(G_1)\times \XBL$. 
This further implies that for any $k\in\Z$ the element $x_e=R(i\gw_k,\Aemo )B_e\phi_k$ is 
obtained with $x_e=Q_e\inv \hat{x}_e$ from the solution of the triangular system 
\ieq{
(i\gw_k-\hat{A}_e^e)\hat{x}_e = Q_eB_e\phi_k
}
 where $\hat{A}_e^e$ is the operator $\hat{A}_e$ with domain $\Dom(\hat{A}_e^e)=\XBBd\times \Dom(G_1)\times \Dom(\CL)$.
A direct estimate
using exponential stability of $(\Amo +B\KtL)\vert_X$ and $(\Amo +L_1\CL)\vert_X$ and the admissibility properties of the operators $B$, $B_d$, $C$, $L_1$, and $K_2$, 
shows that for all $k\in\Z$ we have
\eq{
\MoveEqLeft \norm{R(i\gw_k,\hat{A}_e^e)Q_eB_e\phi_k}
\lesssim \max \bigl\{
\norm{E\phi_k},\norm{F\phi_k},\\
&\hspace{1.2cm} \norm{R(i\gw_k,G_1+G_2C_1)G_2}(\norm{\CL R(i\gw_k,A_L)B_d}\norm{E\phi_k}+\norm{F\phi_k}) 
\bigr\}.
}
The condition~\eqref{eq:ContrOneEFcond} for the solvability of the robust output regulation problem now follows from
 $\norm{R(i\gw_k,G_1+G_2C_1)G_2}=\norm{(P_L(i\gw_k)K_{1k})\inv}$.

 Finally, we will show that under the additional assumptions on the perturbations the conditions~\eqref{eq:ContrTwoEFcondpert} are equivalent to $(\norm{R(i\gw_k,\Aetmo )\tilde{B}_e\phi_k})_{\abs{k}\geq N}\in \lp[2](\C)$. Due to the assumption $\set{i\gw_k}_{k\in\Z}\subset \rho(\tilde{A}_e)$ this is further equivalent to $(\norm{R(i\gw_k,\Aetmo )\tilde{B}_e\phi_k})_{k\in\Z}\in \lp[2](\C)$. 
Let $k\in\Z$ be such that $\abs{k}\geq N$.
  We begin by characterizing $\ker(i\gw_k-\mc{G}_1)$. Let $z_k = (z_1^k,x_1^k)^T\in \ker(i\gw_k-\mc{G}_1)$ for some $k\in\Z$ and denote $C_K = \CL + D\KtL$. Then using the fact that $(G_1,G_2)$ satisfy the \Gconds\ in Definition~\ref{def:Gconds} we get
\eq{
\pmat{(i\gw_k-G_1)z_1^k - G_2C_K x_1^k \\(i\gw_k-\Amo -B\KtL)x_1^k -LC_Kx_1^k} = \pmat{0\\0} 
\quad
\Leftrightarrow
\quad
\left\{
\begin{array}{l}
  C_K x_1^k = 0\\
  (i\gw_k-G_1)z_1^k=0\\
  (i\gw_k-\Amo -B\KtL)x_1^k=0
\end{array}
\right.
}
and since $i\gw_k\in \rho( (\Amo +B\KtL)\vert_X)$, we have $z_k = (z_1^k,0)^T$ where $z_1^k\in \ker(i\gw_k-G_1)$. This immediately implies that the restriction of the operator $\tilde{P}(i\gw_k)K$ to the subspace $\ker(i\gw_k-\mc{G}_1)$ is given by
$(\tilde{P}(i\gw_k)K)\vert_{\ker(i\gw_k-\mc{G}_1)} = \tilde{P}(i\gw_k)K_{1k}$ and it is boundedly invertible by assumption.
If we denote
$\tilde{y}_k = \tilde{P}_d(i\gw_k)\tilde{E}\phi_k + \tilde{F}\phi_k$,  Lemma~\ref{lem:RBeformgeneral} implies
\eq{
R(i\gw_k,\Aetmo )\tilde{B}_e\phi_k
  &= \pmat{
  R(i\gw_k,\tilde{A})(
  \tilde{B}_d\tilde{E}\phi_k-
  \tilde{B}K_{1k}(\tilde{P}(i\gw_k)K_{1k})\inv \tilde{y}_k)
  \\\tilde{z}_k\\0
}
}
where $\tilde{z}_k=(\tilde{z}_k^l)_{l\in\Z}\in \ZI$ is such that $\tilde{z}_k^k =-(\tilde{P}(i\gw_k)K_{1k})\inv \tilde{y}_k$ and $\tilde{z}_k^l=0$ for all $l\neq k$.
This immediately implies that $(\norm{R(i\gw_k,\Aetmo )\tilde{B}_e\phi_k})_{\abs{k}\geq N}\in \lp[2](\C)$ if and only if~\eqref{eq:ContrOneEFcondpert} are satisfied.
\hfill$\square$

\subsection{Controller with a Reduced Order Internal Model}
\label{sec:ContrROIM}

In this section we modify the internal model in the controller in Section~\ref{sec:ContrOne} to design a controller that is robust with respect to a predefined class $\Ops_0\subset \Ops$ of perturbations. 
In this section we assume that $\set{i\gw_k}_{k\in\Z}\subset \rho(A)$ and $P(i\gw_k)$ are boundedly invertible for all $k\in\Z$.
In particular we then have that either both $U$ and $Y$ are infinite-dimensional, or $\dim Y=\dim U$.
The class of admissible perturbations $\Ops_0\subset \Ops$ 
may be chosen freely, but it is assumed that all its perturbations
$\PARsysopspert\in\Ops_0$
are such that
$i\gw_k\in \rho(\tilde{A})$ and $\tilde{P}(i\gw_k)$ is boundedly invertible for all $k\in \Z$.
For such $\Ops_0$ the construction of the reduced order internal model begins by defining
\eq{
\mc{S}_k = \Span \Bigl\{\tilde{P}(i\gw_k)\inv & (\tilde{P}_d(i\gw_k)\tilde{E}\phi_k+\tilde{F}\phi_k) \Bigm| 
\PARsysopspert\in \Ops_0 \Bigr\} \subset U
}
and $p_k = \dim \mc{S}_k$ for $k\in \Z$.
 The number of copies of each frequency $i\gw_k$
 of the exosystem that we include in the 
 reduced order internal model  is equal to~$p_k$.
Let
\eq{
\ZI = \Setm{(z_k)_{k\in\Z}}{z_k\in Y_k \quad \forall k \quad \mbox{and} \quad \sum_{k\in\Z}\norm{z_k}^2<\infty}
}
where $Y_k = \C^{p_k}$ if $p_k<\dim Y$ and $Y_k=Y$ if $p_k=\dim Y$ or $p_k=\infty$. Define $G_1: \Dom(G_1)\subset \ZI\to \ZI$ by
\eq{
G_1 = \diag \bigl( i\gw_k I_{Y_k} \bigr)_{k\in\Z}, 
\quad 
\Dom(G_1) = \Setm{(z_k)_{k\in\Z}\in \ZI}{\sum_{k\in\Z}\abs{\gw_k}^2 \norm{z_k}^2<\infty},
}
and choose $K_1\in \Lin(\ZI,U)$ 
such
that $K_1  =( K_{1k})_{k\in\Z}$
with
\eq{
K_{1k} = 
\left\{
\begin{array}{ll}
  \gg_k
Q_k
  \in \Lin(\C^{p_k},U) &\mbox{if}~~ p_k<\dim Y\\[1.5ex]
  \displaystyle \gg_k\frac{P(i\gw_k)\inv}{\norm{P(i\gw_k)\inv}}
  \in \Lin(Y,U) &\mbox{if}~~ p_k=\dim Y ~ \mbox{or} ~ p_k=\infty.
\end{array}
\right.
}
Here
$Q_k = [u_k^1,\ldots,u_k^{p_k}]\in \Lin(\C^{p_k},U)$
where $\set{u_k^l}_{l=1}^{p_k}\subset U$ are bases of the subspaces $\mc{S}_k$ normalized in such a way that 
$\sup_k\norm{Q_k}_{\Lin(\C^{p_k},U)}<\infty$, 
and 
$(\gg_k)_{k\in\Z}\in \lp[2](\C)$ with $\gg_k> 0$. We then have that $(\norm{K_{1k}})_{k\in\Z}\in \lp[2](\C)$ and thus
 $K_1\in \Lin(\ZI,U)$.
Finally, we define
  \ieq{
  G_2 = (-(P_L(i\gw_k)K_{1k})^\ast)_{k\in\Z} \in \Lin(Y,\ZI).
  }
  The structure and the rest of the parameters of $(\mc{G}_1,\mc{G}_2,K)$ are chosen as in the beginning of Section~\textup{\ref{sec:ContrOne}}.

\begin{thm}
  \label{thm:ContrOneROIM}
  Assume
$P(i\gw_k)$ are invertible for all $k\in \Z$. 
Assume further that $E\in \Lin(W,\Uw)$ and $F\in \Lin(W,Y)$
satisfy%
\begin{subequations}%
  \label{eq:ContrOneROIMEFcond}
  \eqn{
  &\left( \norm{R(i\gw_k,\Amo )\left( B_dE\phi_k- BP(i\gw_k)\inv y_k \right)} \right)_{k\in\Z} \in \lp[2](\C)\\
  &\left(\gg_k\inv \norm{Q_k\pinv P(i\gw_k)\inv y_k}\right)_{k\in \Jinds} \in \lp[2](\C) \\
  &\left(\gg_k\inv \norm{P(i\gw_k)\inv}\norm{ y_k}\right)_{k\in\Z\setminus \Jinds} \in \lp[2](\C) 
  } 
\end{subequations}%
where
 $y_k = P_d(i\gw_k)E\phi_k + F\phi_k$ and $J\subset \Z$ is the set of indices for which $Y_k\neq Y$.
  Then the controller 
  solves the robust output regulation problem 
for the class $\Ops_0$ of perturbations.

In particular, the controller is robust with respect to all perturbations in~$\Ops_0$ 
for which the strong closed-loop stability is preserved, $\set{i\gw_k}_{k\in\Z}\subset\rho(\tilde{A}_e)$, and%
\begin{subequations}%
  \label{eq:ContrOneROIMEFcondpert}
  \eqn{
  &\left( \norm{R(i\gw_k,\Atmo )\left( \tilde{B}_d\tilde{E}\phi_k- \tilde{B}\tilde{P}(i\gw_k)\inv \tilde{y}_k \right)} \right)_{k\in\Z} \in \lp[2](\C)\\
  &\left(\gg_k\inv \norm{Q_k\pinv \tilde{P}(i\gw_k)\inv \tilde{y}_k}\right)_{k\in \Jinds} \in \lp[2](\C) \\
  &\left(\gg_k\inv \norm{P(i\gw_k)\inv}\norm{P(i\gw_k)\tilde{P}(i\gw_k)\inv \tilde{y}_k}\right)_{k\in\Z\setminus \Jinds} \in \lp[2](\C) 
  } 
\end{subequations}%
where
 $\tilde{y}_k = \tilde{P}_d(i\gw_k)\tilde{E}\phi_k + \tilde{F}\phi_k$.
\end{thm}

\begin{proof}
  The strong stability of the closed-loop system and the property $i\R\subset \rho(A_e)$ can be verified similarly as in the proof of Theorem~\ref{thm:ContrOneMain} since we again have $C_1 = \CL H+DK_1 = -G_2^\ast$
and the components $G_{2k}\in \Lin(Y,Y_k)$ of $G_2$ are either boundedly invertible or surjective finite rank operators.

Let $\PARsysopspert\in \Ops_0$ and $k\in \Z$ be arbitrary and denote
$\tilde{y}_k  = \tilde{P}_d(i\gw_k)\tilde{E}e_k +\tilde{F}e_k$.
We begin by showing that
the equations~\eqref{eq:RORPchareqns} in Theorem~\ref{thm:RORPchar} have a solution for every $k\in\Z$, i.e.,
  there exists $z_k \in  \ker(i\gw_k-\mc{G}_1)$ such that
    \ieq{
    \tilde{P}(i\gw_k)\KL z_k = -\tilde{y}_k .
    }
    If $k\in\Z$ is such that $Y_k = Y$, we can choose $z_k=\left[ z_{1k}\atop 0 \right]$ with $z_{1k}=(z_{1k}^l)_{l\in\Z}\in \ZI$ such that $z_{1k}^l=0$ for $l\neq k$ and $z_{1k}^k = -\frac{1}{\gg_k}\norm{P(i\gw_k)\inv} P(i\gw_k)\tilde{P}(i\gw_k)\inv \tilde{y}_k $. Then clearly $\left[ z_{1k}\atop 0 \right]\in \ker(i\gw_k-\mc{G}_1)$ and 
    \eq{
    \tilde{P}(i\gw_k)\KL \pmat{z_{1k}\\0} =  \tilde{P}(i\gw_k)K_{1k}z_{1k}
    = -\tilde{P}(i\gw_k)\tilde{P}(i\gw_k)\inv \tilde{y}_k =-\tilde{y}_k .
    }
It remains to consider the situation
$Y_k\neq Y$.
We have $\tilde{P}(i\gw_k)\inv \tilde{y}_k \in \mc{S}_k = \ran(Q_k)$ by definition. Choose $z_k = \left[ z_{1k}\atop 0 \right]$ with $z_{1k}=(z_{1k}^l)_{l\in\Z}$ such that $z_{1k}^l=0$ for all $l\neq k$ and $z_{1k}^k = - \frac{1}{\gg_k} Q_k\pinv \tilde{P}(i\gw_k)\inv \tilde{y}_k  \in Y_k = \C^{p_k}$.
Then clearly
$\left[ z_{1k}\atop 0 \right]\in \ker(i\gw_k-\mc{G}_1)$ and
  \eq{
  \tilde{P}(i\gw_k) \KL \pmat{z_{1k}\\0} 
  = \tilde{P}(i\gw_k) K_{1k} z_{1k}^k
  =-\tilde{P}(i\gw_k) Q_k Q_k\pinv \tilde{P}(i\gw_k)\inv \tilde{y}_k 
=-\tilde{y}_k .
  }
  Thus the equations~\eqref{eq:RORPchareqns} in Theorem~\ref{thm:RORPchar} have a solution for all $k\in\Z$.

  To prove that the Sylvester equation $\Sigma S = \Aetmo \Sigma + \tilde{B}_e$ has a solutions satisfying $\ran(\Sigma)\subset \Dom(\CeLt)$ whenever~\eqref{eq:ContrOneROIMEFcondpert} holds, we will construct the solution $\Sigma = (\Pi,\Gamma)^T$ explicitly. Let $k\in\Z$, choose $z_k$ as above and define $\Pi \phi_k = R(i\gw_k,\Atmo)(\tilde{B}\KL z_k + \tilde{B}_d\tilde{E}\phi_k)$ and $\Gamma \phi_k =z_k$. 
  Then the properties $\Gamma \phi_k = z_k\in \ker(i\gw_k-\mc{G}_1)$ and $\tilde{P}(i\gw_k)\KL z_k = -\tilde{y}_k  = -\tilde{P}_d(i\gw_k) \tilde{E}\phi_k - \tilde{F}\phi_k$ imply
  \eq{
  (i\gw_k-\tilde{A}_e) \Sigma\phi_k
  &= \pmat{(i\gw_k-\tilde{A})\Pi\phi_k - \tilde{B}\KL\Gamma \phi_k \\
  (i\gw_k-\mc{G}_1)\Gamma \phi_k - \mc{G}_2(\CLt\Pi\phi_k + \tilde{D}\KL\Gamma \phi_k)}
  = \tilde{B}_e\phi_k,
  }
  or $\Sigma S\phi_k = \tilde{A}_e\Sigma \phi_k+ \tilde{B}_e\phi_k$. By Lemma~\ref{lem:Sylsolvabilitycond} we have $\Sigma \in \Lin(W,X_e)$ and $\ran(\Sigma )\subset \Dom(\CeLt)$ if $(\norm{\Sigma\phi_k})_{k\in\Z}=(\norm{R(i\gw_k,\tilde{A}_e)\tilde{B}_e\phi_k})_{k\in\Z}\in\lp[2](\C)$.
If $k\in\Z$ is such that $ Y_k = Y$, then 
\eq{
\norm{\Sigma \phi_k}
&\lesssim \norm{\Pi \phi_k} + \norm{\Gamma \phi_k}
= \norm{R(i\gw_k,\tilde{A})(\tilde{B} K_{1k}z_{1k}^k + \tilde{B}_d\tilde{E}\phi_k)}  + \norm{z_{1k}^k}\\
&= \norm{ R(i\gw_k,\tilde{A})(\tilde{B}_d\tilde{E}\phi_k - \tilde{B} \tilde{P}(i\gw_k)\inv \tilde{y}_k ) }  + \frac{  \norm{P(i\gw_k) \tilde{P}(i\gw_k)\inv \tilde{y}_k  }}{\gg_k \norm{P(i\gw_k)\inv}\inv}
} 
On the other hand, if $k\in\Z$ is such that $ Y_k = \C^{p_k}$, then 
\eq{
\MoveEqLeft \norm{\Sigma \phi_k}
\lesssim
 \norm{R(i\gw_k,\tilde{A})(\tilde{B} K_{1k}z_{1k}^k + \tilde{B}_d\tilde{E}\phi_k)}  + \norm{z_{1k}^k}\\
 &=\norm{ R(i\gw_k,\tilde{A})(\tilde{B}_d\tilde{E}\phi_k - \tilde{B} \tilde{P}(i\gw_k)\inv \tilde{y}_k ) 
 }  + \frac{\norm{Q_k\pinv \tilde{P}(i\gw_k)\inv \tilde{y}_k }}{\gg_k}.
}
Thus $(\norm{\Sigma\phi_k})_{k\in\Z}\in\lp[2](\C)$ is satisfied if~\eqref{eq:ContrOneROIMEFcondpert} are satisfied. In the case of the nominal plant the conditions~\eqref{eq:ContrOneROIMEFcondpert} reduce to~\eqref{eq:ContrOneROIMEFcond}.  
\end{proof}

\subsection{A New Controller Structure for Output Regulation}
\label{sec:ContrOneORP}

We can also use the new controller structure to achieve output tracking and disturbance rejection in the situation where the controller is not required to be robust.
In this section we assume $\set{i\gw_k}_{k\in\Z}\subset \rho(A)$, but it is not required that Assumption~\ref{ass:PKsurj} on the surjectivity of the transfer functions is satisfied. Instead, we only need to assume that 
$y_k\in \ran(P(i\gw_k))$ for all $k\in\Z$, where $y_k = P_d(i\gw_k)E\phi_k + F\phi_k$.

Let $\mc{G}_1$ to have the same general structure as in Section~\ref{sec:ContrOne}. Choose $\ZI = W = \lp[2](\C)$ and $G_1 = S = \diag(i\gw_k)_{k\in\Z}$, and choose $K_1 
= (K_{1k})_{k\in\Z}
: \ZI\to U$ 
with
$K_{1k}=\gg_k \frac{u_k}{\norm{u_k}}$ 
where
$(\gg_k)_{k\in\Z}\in \lp[2](\C)$ with $\gg_k\neq 0$ and
where $u_k\in U$ are 
such that
\eq{
\left\{
\begin{array}{ll}
  P(i\gw_k)u_k =y_k \qquad & \mbox{if} \quad y_k\neq 0\\[1ex]
  u_k \notin \ker(P(i\gw_k))  \quad\qquad & \mbox{if} \quad y_k= 0.
\end{array}
\right.
}
where $y_k = P_d(i\gw_k)E\phi_k + F\phi_k$.
If $P(i\gw_k)$ has a closed range, we can in particular choose $u_k=P(i\gw_k)\pinv y_k$.
We have $\left( \norm{K_{1k}} \right)_{k\in\Z} \in \lp[2](\C)$ and thus $K_1\in \Lin(\ZI,U)$.
We define
  \ieq{
  G_2 = (-(P_L(i\gw_k)K_{1k})^\ast)_{k\in\Z} \in \Lin(Y,\ZI).
  }
  The rest of the parameters of $(\mc{G}_1,\mc{G}_2,K)$ are chosen as in the beginning of Section~\textup{\ref{sec:ContrOne}}.

\begin{thm}
  \label{thm:ContrOneORP}
  Assume
$E\in \Lin(W,\Uw)$ and $F\in \Lin(W,Y)$ satisfy%
\begin{subequations}%
  \label{eq:ContrOneORPEFcond}
  \eqn{
  &\left( \norm{ R(i\gw_k,A)B_dE\phi_k - R(i\gw_k,A)Bu_k } \right)_{k\in J} \in \lp[2](\C)\\
  &\left(\gg_k\inv \norm{u_k}\right)_{k\in J} \in \lp[2](\C) \qquad \mbox{and} \qquad 
   \left( \norm{R(i\gw_k,A)B_dE\phi_k} \right)_{k\in\Z\setminus J} \in \lp[2](\C)
  } 
\end{subequations}%
  where $y_k = P_d(i\gw_k)E\phi_k + F\phi_k$ and where $J\subset \Z$ is the set of indices for which $y_k\neq 0$.
  Then the controller 
solves the output regulation problem.
\end{thm}

\begin{proof}
  We have $P_L(i\gw_k)K_{1k}\neq 0$ for all $k\in\Z$. The strong stability of the closed-loop system
  and $i\R\subset \rho(A_e)$ 
  can be shown similarly as in the proof of Theorem~\ref{thm:ContrOneMain} since 
  $C_1=\CL H+DK_1 = -G_2^\ast$, and
  $G_{2k}\in \Lin(Y,\C)$
are rank one and nonzero.
  
  It remains to show that the regulator equations~\eqref{eq:regeqns} have a solution $\Sigma \in \Lin(W,X_e)$ satisfying $\ran(\Sigma)\subset \Dom(\CeL)$.
We will construct the solution $\Sigma = (\Pi,\Sigma)^T$ 
explicitly by defining $\Pi\phi_k$ and $\Gamma \phi_k$.  
Let $k\in\Z$. If $y_k\neq 0$, then define $\Gamma \phi_k = \left[ z_{1k}\atop 0 \right]$ where
$z_{1k}=(z_{1k}^l)_{l\in\Z}$ with $z_{1k}^l=0$ for all $l\neq k$ and
$z_{1k}^k= -\frac{\norm{u_k}}{\gg_k}$.
Then $\Gamma\phi_k\in \ker(i\gw_k-\mc{G}_1)$. Moreover, define $\Pi\phi_k = R(i\gw_k,A)(BK\Gamma \phi_k +B_dE\phi_k)\in \XBBd$.
Then
\eq{
\CeL \Sigma \phi_k + D_e\phi_k 
= \CL \Pi\phi_k + D\KL \Gamma \phi_k +F\phi_k 
= -P(i\gw_k)u_k  + y_k
=0.
}
Moreover, since $\Gamma\phi_k\in \ker(i\gw_k-\mc{G}_1)$ we have that $\CL \Pi\phi_k + D\KL \Gamma \phi_k =-F\phi_k$ implies
\eq{
(i\gw_k-A_e)\Sigma \phi_k 
&= \pmat{(i\gw_k-A)\Pi\phi_k - B\KL\Gamma \phi_k \\
(i\gw_k-\mc{G}_1)\Gamma \phi_k - \mc{G}_2(\CL\Pi\phi_k + D\KL\Gamma \phi_k)}
= B_e\phi_k,
}
or equivalently $\Sigma S\phi_k = \Aemo \Sigma \phi_k + B_e\phi_k$.
Alternatively, if $k\in\Z$ is such that $y_k = 0$, we choose $\Gamma\phi_k = 0$ and $\Pi\phi_k = R(i\gw_k,A)B_dE\phi_k$. Then we can similarly see that $\CeL\Sigma \phi_k + D_e\phi_k = 0$ and $\Sigma S\phi_k = \Aemo \Sigma\phi_k + B_e\phi_k$.
It remains to show that $\Sigma$ is bounded and $\ran(\Sigma)\subset \Dom(\CeL)$. By Lemma~\ref{lem:Sylsolvabilitycond} the operator $\Sigma$ has these properties  if $(\norm{\Sigma\phi_k})_{k\in\Z}=(\norm{R(i\gw_k,\Aemo)B_e\phi_k})_{k\in\Z}\in \lp[2](\C)$.  
If $y_k\neq 0$, we have
\eq{
\norm{\Sigma \phi_k} 
&\lesssim \norm{\Pi\phi_k}+ \norm{\Gamma \phi_k}
= \norm{R(i\gw_k,A)(BK_{1k}z_{1k}^k + B_dE\phi_k)} + \norm{z_{1k}^k}\\
&= \norm{ R(i\gw_k,A)B_dE\phi_k - R(i\gw_k,A)Bu_k } + \frac{\norm{u_k}}{\gg_k}.
}
Similarly, if $y_k=0$, then
\eq{
\norm{\Sigma \phi_k} 
\lesssim 
 \norm{R(i\gw_k,A)(BK_{1k}z_{1k}^k + B_dE\phi_k)} + \norm{z_{1k}^k}
= \norm{ R(i\gw_k,A)B_dE\phi_k } .
} 
We thus have $(\norm{\Sigma \phi_k})_{k\in\Z}\in \lp[2](\C)$ due to the conditions~\eqref{eq:ContrOneORPEFcond}. By Theorem~\ref{thm:ORP} the controller solves the output regulation problem.
\end{proof}

\section{The Observer-Based Controller}
\label{sec:ContrTwo}

In this section we construct an observer based controller that solves the robust output regulation problem. 
For this controller structure it is necessary to assume that $P_K(i\gw_k)$ are boundedly invertible for all $k\in\Z$. This in particular implies $\dim U = \dim Y$ if $\dim Y<\infty$.

\medskip

\noindent\textbf{Step $\bm{1}^\circ$:}
We begin by choosing the state space of the controller as $Z=\ZI\times X$, and choosing the structure of $\PARcontr$ as
  \eq{
  \mc{G}_{1}=\pmat{G_1&0\\(B+LD)\KoL &\Amo +B\KtL +L(\CL +D\KtL)}, 
\quad~  \mc{G}_{2}=\pmat{G_2\\-L},
  }
  and $K = (\KoL, \; \KtL)$. 
  By Assumption~\ref{ass:stabilizability} we can choose
$K_{21}\in\Lin(X_1,U)$ and $L\in\Lin(Y,X_{-1})$ in such a way that
$(\Amo +B\KtoL)\vert_X$ and $(A+L\CL)\vert_X $ generate exponentially stable semigroups and the system $\left( A,[B,~ L, ~B_d], \pmatsmall{C\\K_{21}},D \right)$ is regular. 
We define
 \eq{
P_K(\gl) = (\CL + D\KtoL) R(\gl,\Amo +B\KtoL)B + D, \qquad \gl\in \rho( \Amo +B\KtoL).
 }
We assume that $P_K(i\gw_k)$ are boundedly invertible for all $k\in\Z$.

\medskip

\noindent\textbf{Step $\bm{2}^\circ$:} 
  The operators $(G_1,G_2)$ make up the internal model of the exosystem~\eqref{eq:exointro}, and they are defined by choosing 
  $\ZI = \lp[2](Y)$, and
\eq{
G_1 = \diag \bigl( i\gw_k I_Y\bigr)_{k\in\Z}, \qquad \Dom(G_1) = \setm{(z_k)_{k\in\Z}\in \ZI}{(\gw_kz_k)_{k\in\Z}\in \lp[2](Y)}
}
We choose 
$G_2=(G_{2k})_{k\in\Z}$ in such a way that
the components $G_{2k}\in \Lin(Y)$ are boundedly invertible and $(\norm{G_{2k}})_{k\in\Z}\in \lp[2](\C)$. In particular, it is 
possible to choose $G_{2k}=g_{2k} I_Y$, where $(g_{2k})_{k\in\Z}\in \lp[2](\C)$ and $g_{2k}\neq 0$ for all $k\in\Z$.
For more concrete choices of $G_{2k}$, see Corollary~\ref{cor:ContrTwoPolExp}.

\medskip

\noindent\textbf{Step $\bm{3}^\circ$:} 
We define  
$H: \ran( [B,~B_d])\subset X_{-1}\to \ZI$
such that
\eq{
Hx = (G_{2k}(\CL+D\KtoL )R(i\gw_k,\Amo +B\KtoL )x)_{k\in\Z}
\qquad
\forall  x
\in
\ran( [B,B_d]).
}
Since we have from~\cite[Sec. 7]{Wei94} that $(A+B\KtoL ,B,B_d,\CL+D\KtoL ,D)$ is a regular linear system, it is immediate that $H$ is well-defined and  $H\in \Lin(X,\ZI)$.

\medskip

\noindent\textbf{Step $\bm{4}^\circ$:}
We choose the operator $K_1\in \Lin(\ZI,U)$ in such a way that
\eq{
K_1 = -\sum_{k\in\Z} (G_{2k} P_K(i\gw_k))^\ast z_k
}
for all $z=(z_k)_{k\in\Z}\in \ZI$.
We define $K_2=\KtoL  + K_1H\in \Lin(X_1,U)$, 
which is an admissible observation operator for $A$ and $\KtL = K_2$. Finally, we 
 choose the domain of the operator $\mc{G}_1$ as
\eq{
\Dom(\mc{G}_1) 
= \Bigl\{
\pmat{z_1\\x_1}
&\in \Dom(G_1)\times \XBL \Bigm|
 (\Amo +L\CL)x_1+(B+LD)(K_1z_1+\KtL x_1)\in X\Bigr\}.
}

\begin{thm}
  \label{thm:ContrTwoMain}
If $E\in \Lin(W,\Uw)$ and $F\in \Lin(W,Y)$ satisfy
  \eqn{
  \label{eq:ContrTwoEFcond}
  \left( \norm{P_K(i\gw_k)\inv G_{2k}\inv}^2 (\norm{E\phi_k}+\norm{F\phi_k}) \right)_{k\in\Z}\in \lp[2](\C),
  }
then the controller solves the robust output regulation problem.

The controller is guaranteed to be robust with respect to all perturbations in $\Ops$ for which the strong closed-loop stability is preserved,
$\set{i\gw_k}_{k\in\Z}\subset\rho(\tilde{A}_e)$, $\set{i\gw_k}_{\abs{k}\geq N}\subset\rho(\tilde{A})$ for some $N\in\N$, $\tilde{P}(i\gw_k)$
are invertible whenever  $\abs{k}\geq N$, and for which%
\begin{subequations}%
  \label{eq:ContrTwoEFcondpert}
  \eqn{
  &\left( \norm{R(i\gw_k,\Atmo )\left( \tilde{B}_d\tilde{E}\phi_k- \tilde{B}\tilde{P}(i\gw_k)\inv \tilde{y}_k \right)} \right)_{\abs{k}\geq N} \in \lp[2](\C)\\
  &\left(\norm{
  (G_{2k}P_K(i\gw_k)^\ast)\inv
  (I-\KtL R_L^k B_L) \tilde{P}(i\gw_k)\inv \tilde{y}_k}\right)_{\abs{k}\geq N} \in \lp[2](\C) \\
  &\left(\norm{R_L^k B_L\tilde{P}(i\gw_k)\inv \tilde{y}_k}\right)_{\abs{k}\geq N} \in \lp[2](\C) 
  } 
\end{subequations}%
where
$\tilde{y}_k = \tilde{P}_d(i\gw_k)\tilde{E}\phi_k + \tilde{F}\phi_k$, $R_L^k = R(i\gw_k,\Amo +L\CL)$ and $B_L = B+LD$.
\end{thm}

If $\set{i\gw_k}_{k\in\Z}\subset \rho(A)$,~\eqref{eq:ContrTwoEFcond} can again be replaced with~\eqref{eq:ContrTwoEFcondpert} for the nominal plant.

\begin{cor}
  \label{cor:ContrTwoAltEF}
  If there exists $N\in\N$ such that $\set{i\gw_k}_{\abs{k}\geq N}\subset \rho(A)$ and
  $\vspace{-1.6ex}\displaystyle\sup_{\abs{k}\geq N}\norm{R(i\gw_k,A)}<\infty$,
  then the conclusions of Theorem~\textup{\ref{thm:ContrOneMain}} hold if 
\eq{
&\left( \norm{ (G_{2k}P_K(i\gw_k)^\ast)\inv }
  \norm{P(i\gw_k)\inv}(\norm{P_d(i\gw_k) E\phi_k}  + \norm{F\phi_k} ) \right)_{\abs{k}\geq N} \in \lp[2](\C)
    }
    and the controller is guaranteed to be robust with respect to perturbations in $\Ops$ for which the strong closed-loop stability is preserved, 
$\set{i\gw_k}_{k\in\Z}\subset\rho(\tilde{A}_e)$,
    $\set{i\gw_k}_{\abs{k}\geq N}\subset \rho(\tilde{A})$,
    $\sup_{\abs{k}\geq N} \norm{R(i\gw_k,\Atmo )}<\infty$,
and
    \eq{
    \left( \norm{ (G_{2k}P_K(i\gw_k)^\ast)\inv } \norm{\tilde{P}(i\gw_k)\inv}(\norm{\tilde{P}_d(i\gw_k) \tilde{E}\phi_k}  + \norm{\tilde{F}\phi_k}) \right)_{\abs{k}\geq N} \in \lp[2](\C).
    }
\end{cor}

\begin{cor}
  \label{cor:ContrTwoPolExp}
  The following hold. 
  \begin{itemize}
    \item[\textup{(a)}] If 
$\norm{P_K(i\gw_k)\inv}=\Omi(\abs{\gw_k}^\ga) $ for some $\ga>0$ and $(\gg_k)_{k\in\Z}\in \lp[2](\C)$, the choice 
      \ieq{
      G_{2k} = \gg_k I_Y
      }
      solves the robust output regulation problem for $E$ and $F$ satisfying
      \eq{
      \left( \frac{\abs{\gw_k}^{2\ga} }{\abs{\gg_k}^2} (\norm{E\phi_k}+\norm{F\phi_k}) \right)_{k\in\Z}\in \lp[2](\C).
      }
      If we in particular choose $\gg_k = \abs{\gw_k}^{-\gb}$  for some $\gb>1/2$ whenever $\gw_k\neq 0$, then $\abs{\gw_k}^{2\ga} /\abs{\gg_k}^2 = \abs{\gw_k}^{2(\ga+\gb)}$ whenever $\gw_k\neq 0$.
      
    \item[\textup{(b)}] If $\norm{P_K(i\gw_k)\inv}=\Omi(e^{\ga\abs{\gw_k}}) $ for some $\ga>0$ and $(\gg_k)_{k\in\Z}\in\lp[2](\C)$, the choice $G_{2k} = \gg_k I_Y$
solves the robust output regulation problem for $E$ and $F$ satisfying
      \eq{
      \left( \frac{e^{2\ga\abs{\gw_k}} }{\abs{\gg_k}^2} (\norm{E\phi_k}+\norm{F\phi_k}) \right)_{k\in\Z}\in \lp[2](\C).
      }
  \end{itemize}
\end{cor}

\noindent\textit{Proof of Theorem~\textup{\ref{thm:ContrTwoMain}}.}
The proof can be completed similarly as the proof of Theorem~\ref{thm:ContrOneMain}. The regularity of the controller follows from~\cite{Wei94}, and \Gconds\ can be verified similarly as in~\cite[Thm. 15]{Pau16a}.
The operator $B_1= HB+G_2D\in \Lin(U,\ZI)$ is well-defined and $B_1=(B_{1k})_{k\in\Z} = (G_{2k} P_K(i\gw_k))_{k\in\Z}$.
Thus $B_1 = -K_1^\ast$ and Lemma~\ref{lem:G1fbstab} implies that 
the semigroup generated by $G_1+B_1K_1=G_1-B_1 B_1^\ast$ is strongly stable, $i\R\subset \rho(G_1+B_1K_1)$, and
\ieq{
 \norm{K_1 R(i\gw_k,G_1+B_1K_1)}
=\norm{ P_K(i\gw_k)\inv G_{2k}\inv }
}
for all $k\in\Z$.  

If we choose a similarity transform
\eq{
Q_{e}=\pmat{-I&0&0\\H&I&0\\-I&0&I} =  Q_{e}\inv 
}
and use the fact that $H$ is a solution of the Sylvester equation $G_1H =H(\Amo +B\KtoL )+G_2(\CL + D\KtoL )$, then we can see that $\hat{A}_e = Q_{e}A_{e}Q_{e}\inv$ where
\eq{
\hat{A}_e= \pmat{\Amo  + B \KtoL  & -BK_1 & -B\KtL  \\ 0& G_1  +B_1K_1 &B_1\KtL   \\ 0&0& \Amo  + L\CL } 
}
with the natural domain.
The properties of $G_1+B_1K_1$ and the exponential stability of $(\Amo +B\KtoL )\vert_X$ and $(A+L\CL)\vert_X $ imply that
 the closed-loop system is strongly stable and $i\R\subset \rho(A_e)$.
Later in Lemma~\ref{thm:IMstabnonuniform} it is shown that
$\norm{R(i\gw_k,G_1-G_2G_2^\ast)}= O(1+\norm{P_K(i\gw_k)\inv G_{2k}\inv }^2) $.
A direct estimate
using the admissibility of  $B$, $B_d$, $C$, $L$, and $K_2$, and the boundedness of $HB_d$ shows that 
$(\norm{R(i\gw_k,\Aemo )B_e\phi_k})_{k\in\Z}\in \lp[2](\C)$ if~\eqref{eq:ContrTwoEFcond} are satisfied.

 Finally, we will show that under the additional assumptions on the perturbations the conditions~\eqref{eq:ContrTwoEFcondpert} are equivalent to
$(\norm{R(i\gw_k,\Aetmo )\tilde{B}_e\phi_k})_{k\in\Z}\in \lp[2](\C)$.
 Let $k\in\Z$ be such that $\abs{k}\geq N$ and denote 
$\tilde{y}=\tilde{P}_d(i\gw_k)\tilde{E}\phi_k + \tilde{F}\phi_k$.
By Lemma~\ref{lem:RBeformgeneral} we have
  \eq{
  R(i\gw_k,\Aetmo )\tilde{B}_e\phi_k
  = \pmat{R(i\gw_k,\tilde{A}) (\tilde{B}_d\tilde{E}\phi_k+ \tilde{B}Kz  )\\z}
}
  where $z\in \ker(i\gw_k-\mc{G}_1)$ is such that
\ieq{
Kz = 
-\tilde{P}(i\gw_k)\inv \tilde{y}.
}
Since $(i\gw_k-\mc{G}_1)z=0$, we have that $z=(z_1,x_1)^T\in \Dom(G_1)\times \XBL$  and denoting $B_L = B+LD$ and 
$R_L = R(i\gw_k,\Amo +L\CL)$ we have
\eq{
&\pmat{i\gw_k-G_1&0\\-B_L\KoL &i\gw_k-\Amo -L\CL - B_L\KtL } \pmat{z_1\\x_1}=0 \\[1ex]
 \Leftrightarrow \qquad &
 \left\{
 \begin{array}{l}
   (i\gw_k-G_1)z_1 = 0\\
   x_1=R_L B_L\KoL z_1+ R_L B_L\KtL x_1 = R_LB_LKz
= -R_L B_L\tilde{P}(i\gw_k)\inv \tilde{y}.
 \end{array}
 \right.
}
Thus $z_1\in \ker(i\gw_k-G_1)$ and $z_1=(z_1^l)_{l\in\Z}$ is such that $z_1^l=0$ for all $l\neq k$ and $K_1z_1 = K_{1k}z_1^k$. 
The above equations imply
\ieq{
K_{1k}z_1^k 
= Kz-\KtL x_1 
= (I-\KtL R_LB_L)Kz
}
and thus
$z_1^k 
= K_{1k}\inv (I-\KtL R_LB_L)Kz
= -K_{1k}\inv (I-\KtL R_LB_L) \tilde{P}(i\gw_k)\inv \tilde{y}
$.
Now
  \eq{
  R(i\gw_k,\Aetmo )\tilde{B}_e\phi_k
  = \pmat{R(i\gw_k,\tilde{A}) ( \tilde{B}_d\tilde{E}\phi_k -\tilde{B}\tilde{P}(i\gw_k)\inv \tilde{y} )\\z_1\\-R_L B_L\tilde{P}(i\gw_k)\inv \tilde{y}
  }
  }
  and $K_{1k} = -(G_{2k}P_K(i\gw_k))^\ast$ and
  \ieq{
  \norm{z_1}
= \norm{
  (G_{2k}P_K(i\gw_k)^\ast)\inv
(I-\KtL R_LB_L) \tilde{P}(i\gw_k)\inv \tilde{y}}
  }
  imply that
 $(\norm{R(i\gw_k,\Aetmo )\tilde{B}_e\phi_k})_{k\in\Z}\in \lp[2](\C)$ if and only if~\eqref{eq:ContrTwoEFcondpert} are satisfied.
 \hfill$\square$

\section[Nonuniform Stability of the Closed-Loop Semigroup]{Nonuniform Stability of the Closed-Loop Semigroup and Decay Rates for $e(t)$}
\label{sec:CLnonuniformstab}

In this section we derive decay rates for the orbits 
$t\mapsto T_e(t)x_{e0}$
of the closed-loop system and the regulation error $e(t)$ corresponding to initial states $x_{e0}\in \Dom(A_e)$ and $v_0\in \Dom(S)$.
These results are based on the theory of nonuniform stability of semigroups developed in~\cite{LiuRao05,BatDuy08,BorTom10,BatChi16}.
It should also be noted that the following theorem is not limited to the controller structures used in this paper, but instead it applies to any robust controller for which the closed-loop system satisfies $\norm{R(i\gw,A_e)}= O(g(\abs{\gw}))$.

\begin{thm}
  \label{thm:CLnonunifstab}
Assume the controller $\PARcontr$ solves the robust output regulation problem
and assume
 $g:\R_+\to[1,\infty)$ is a monotonically increasing function such that
$\norm{R(i\gw,A_e)}\leq O(g(\abs{\gw}))$.
In particular, for the controllers in Sections~\textup{\ref{sec:ContrOne}} and~\textup{\ref{sec:ContrTwo}} we can choose $g(\cdot)$ for which
   there exists $M_g,\gw_g>0$ such that 
   $\abs{\ddb{\gw} \frac{1}{g(\gw)}}\leq M_g$
   for all $\gw>\gw_g$ and 
\begin{itemize}
    \setlength{\itemsep}{1ex}
  \item[\textup{(a)}] $\norm{(P_L(i\gw_k)K_{1k})\inv}^2\leq g(\abs{\gw_k})$ for all $k\in\Z$ 
for the controller in Section~\textup{\ref{sec:ContrOne}}.
  \item[\textup{(b)}] $\norm{P_K(i\gw_k)\inv G_{2k}\inv}^2\leq g(\abs{\gw_k})$ for all $k\in\Z$ 
for the controller in Section~\textup{\ref{sec:ContrTwo}}.
\item[\textup{(c)}] 
 $\norm{(P_L(i\gw_k)K_{1k})\pinv}^2\leq g(\abs{\gw_k})$ for all $k\in\Z$ for the controller in Section~\textup{\ref{sec:ContrROIM}}. 
\item[\textup{(d)}] 
  $\norm{u_k}^2\abs{\gg_k}^{-2}\norm{P_L(i\gw_k)u_k}^{-2}\leq g(\abs{\gw_k})$ for all $k\in\Z$ for the controller in Section~\textup{\ref{sec:ContrOneORP}}.  
\end{itemize}
Then there exist $M_0,M_e>0$ and $0<c<1$ such that 
\eqn{
\label{eq:CLnonuniformstate}
\norm{T_e(t)x_{e0}}\leq \frac{M_e}{\Mlog\inv(ct)} \norm{A_ex_{e0}}, \qquad \forall x_{e0}\in 
\Dom(A_e) 
}
where $\Mlog\inv$ is the inverse of the function
\eq{
\Mlog(\gw) = M_0g(\gw)(\log(1+M_0g(\gw))+\log(1+\gw)), \qquad \gw>0.
}
Furthermore, there exists $M_e^e\geq 1$ such that 
for all $v_0\in \Dom(S)$ and for all initial states of the form $x_{e0} = \xeoo - \Aemo \inv B_ev_0$ where $\xeoo\in \Dom(A_e)$ the regulation error satisfies
\eqn{
\label{eq:CLnonuniformerror}
\int_t^{t+1} \norm{e(s)}ds\leq \frac{M_e^e}{\Mlog\inv(ct)} 
\norm{A_e\xeoo - \Sigma Sv_0}
}
where $\Sigma \in \Lin(W,X_e)$ is the solution of the regulator equations~\eqref{eq:regeqns}.
\end{thm}

The function $\Mlog\inv(t)$ has particularly simple forms in the most important situations where $g(\cdot)$ is either a polynomial or an exponential function, i.e., when the norms
   $\norm{(P_L(i\gw_k)K_{1k})\inv}$ and 
   $\norm{P_K(i\gw_k)\inv G_{2k}\inv}$ grow either polynomially or exponentially fast when $\abs{\gw_k}$ is large.
    In particular, if 
    $g(\gw)=M_0\gw^\ga$
    for some $\ga,M_0>0$, then 
    $\Mlog\inv(ct)\sim (t/\log t)^{1/\ga}$~\cite[Ex. 1.7]{BatChi16},
    and thus~\eqref{eq:CLnonuniformstate} and~\eqref{eq:CLnonuniformerror} simplify to
    \eq{
\norm{T_e(t)x_{e0}} &\leq M_e
\left( \frac{\log t}{t} \right)^{\frac{1}{\ga}} \hspace{-1ex}
\norm{A_ex_{e0}},  \quad
\int_t^{t+1} \hspace{-1ex}\norm{e(s)}ds
\leq 
M_e^e\left( \frac{\log t}{t} \right)^{\frac{1}{\ga}} \hspace{-1ex}
\norm{A_e\xeoo - \Sigma Sv_0}.
    }
    In addition, if $X$ is a Hilbert space, then it follows from~\cite{BorTom10} that the logarithm in the above inequalities can be omitted and we have
    \eq{
\norm{T_e(t)x_{e0}} &\leq 
 \frac{M_e}{t^{1/\ga}}   \norm{A_ex_{e0}}, 
\qquad
\int_t^{t+1} \norm{e(s)}ds
\leq 
\frac{M_e^e }{t^{1/\ga}}
\norm{A_e\xeoo - \Sigma Sv_0}.
    }
Moreover, if
    $g(\gw)=M_0e^{\ga \gw}$ 
    for some $\ga,M_0>0$, then $\Mlog\inv(ct)\sim \frac{1}{\ga}\log t$~\cite[Ex. 1.6]{BatDuy08}
    and~\eqref{eq:CLnonuniformstate} and~\eqref{eq:CLnonuniformerror} simplify to
    \eq{
\norm{T_e(t)x_{e0}}\leq \frac{\ga M_e
}{\log t}\norm{A_ex_{e0}}, 
\qquad
\int_t^{t+1} \norm{e(s)}ds
\leq \frac{\ga M_e^e }{\log t}
\norm{A_e\xeoo - \Sigma Sv_0}.
    }

In the case where the operator $B_d$ and the operator $L$ in the controller are bounded, the convergence rate~\eqref{eq:CLnonuniformerror} is achieved for all $v_0\in \Dom(S)$ and $x_{e0}\in \Dom(A_e)$, since in this case we have $B_e\in \Lin(W,X_e)$ and $\Aemo\inv  B_ev_0\in \Dom(A_e)$. The norm on the right-hand side of~\eqref{eq:CLnonuniformerror} can then be estimated by
\eq{
\norm{A_e\xeoo - \Sigma Sv_0}
\leq \norm{A_ex_{e0}} + \norm{\Sigma}\norm{Sv_0}+\norm{B_e}\norm{v_0}.
}
In the further particular case where also the operator $B$ is bounded, we have $\Dom(A_e) = \Dom(A)\times \Dom(\mc{G}_1) = \Dom(A)\times \Dom(G_1)\times \Dom(A)$ for both of the controller structures. Then an easy estimate using the admissibility of $C$ and $K$ leads to the following corollary.
  Here $\norm{x}_{\Dom(A)} = \norm{Ax}+\norm{x}$ and  $\norm{z_0}_{\mc{G}_1}$ and $\norm{v_0}_{\Dom(S)}$ are defined analogously.

\begin{cor}
  \label{cor:CLnonuniformbddops}
  If $B$, $B_d$, and $L$ are bounded operators and the assumptions of Theorem~\textup{\ref{thm:CLnonunifstab}} are satisfied, then 
  there exists $\tilde{M}_e^e> 0$ such that for all initial states $v_0\in \Dom(S)$, $x_0\in \Dom(A)$ and $z_0\in \Dom(\mc{G}_1)$ 
  \eq{
  \int_t^{t+1} \norm{e(s)}ds\leq \frac{\tilde{M}_e^e}{\Mlog\inv(ct)} 
  \left( \norm{x_0}_{\Dom(A)}+ \norm{ z_0}_{\Dom(\mc{G}_1)}+ \norm{v_0}_{\Dom(S)} \right).
  }
  In particular, for the controllers in Sections~\textup{\ref{sec:ContrOne}},~\textup{\ref{sec:ContrROIM}},~\textup{\ref{sec:ContrOneORP}} and~\textup{\ref{sec:ContrTwo}} 
  there exists $\tilde{M}_e^e> 0$ such that for all initial states $v_0\in \Dom(S)$, $x_0\in \Dom(A)$ and $z_0 = (z_1^0,x_1^0)^T\in \Dom(G_1)\times \Dom(A)$ 
  \eq{
  \int_t^{t+1} \norm{e(s)}ds\leq \frac{\tilde{M}_e^e}{\Mlog\inv(ct)} 
  \left( \norm{x_0}_{\Dom(A)}  + \norm{z_1^0}_{\Dom(G_1)} + \norm{x_1^0}_{\Dom(A)} + \norm{v_0}_{\Dom(S)} \right).
  }
\end{cor}

Finally, if the operators $C$ and $K$ are bounded, then the proof of Theorem~\ref{thm:CLnonunifstab} shows that the regulation error $e(t)$ decays at a rate
\ieq{
\norm{e(t)}
\leq \frac{M_e^e}{\Mlog\inv(ct)} 
\norm{A_e\xeoo - \Sigma Sv_0}.
}

The nonuniform stability properties described in Theorem~\ref{thm:CLnonunifstab} are based on the following result on the behaviour of the resolvent of the stabilized internal model.   Theorem~\ref{thm:IMstabnonuniform} in particular implies that the semigroup generated by $G_1-G_2G_2^\ast$ is nonuniformly stable in the sense of~\cite{BatDuy08,BorTom10,BatChi16}.

\begin{thm}
  \label{thm:IMstabnonuniform}
  Assume $U$ and $Y_k$ for $k\in\Z$ are Hilbert spaces.
  Consider
  $\ZI = \setm{(z_k)_{k\in\Z}\in\mbox{\scalebox{1.4}{$\otimes$}}_{k\in\Z}Y_k}{\sum_{k\in\Z}\norm{z_k}_{Y_k}^2<\infty}$ with inner product $\iprod{z}{v}=\sum_{k\in\Z}\iprod{z_k}{v_k}_{Y_k}$ for $z=(z_k)_k$ and $v=(v_k)_k$. 
Assume $\set{i\gw_k}_{k\in\Z}$ has a uniform gap, and
  let $G_1 = \diag(i\gw_k I_{Y_k})\kZ$ on $\ZI$ 
with domain
  $\Dom(G_1)=\setm{(z_k)\kZ\in \ZI}{(\gw_kz_k)_k\in\ZI}$ and $G_2 = (G_{2k})_{k\in\Z} \in \Lin(U,\ZI)$.

Assume further that $G_{2k}\in \Lin(U,Y_k)$ of $G_2$ are surjective and $G_{2k}^\ast$ have closed ranges for all $k\in\Z$, and assume $g:\R_+\to[1,\infty)$ is a monotonically increasing function such that $\norm{G_{2k}\pinv}^2\leq g(\abs{\gw_k})$ for all $k\in\Z$, and there exists $M_g,\gw_g>0$ such that $\abs{\ddb{\gw} \frac{1}{g(\gw)}}\leq M_g$ for all $\gw>\gw_g$. 
Then the semigroup generated by $G_1-G_2G_2^\ast$ is strongly stable, $i\R\subset \rho(G_1-G_2G_2^\ast)$ and
there exists $M>0$ such that
   \eqn{
   \label{eq:IMstabnonunifResest}
   \norm{R(i\gw,G_1-\BI\BI^\ast)}\leq  
   Mg(\abs{\gw}), \qquad \abs{\gw}>\gw_g.
   }
\end{thm}

\begin{proof}
  The strong stability of the semigroup and 
  $i\R\subset \rho(G_1 - \BI \BI^\ast)$ follow from Lemma~\ref{lem:G1fbstab}.
The approach in the remaining part of the proof is inspired by the technique used in~\cite[Ex. 1--3]{LiuRao05}.
If  the estimate~\eqref{eq:IMstabnonunifResest} does not hold for any $M> 0$, then there exists $(s_n)_{n\in\N}\subset \R$ satisfying $\abs{s_n}\geq 1$, and $\abs{s_n}\to\infty$ as $n\to \infty$, and $(z_n)_{n\in\N}\subset \Dom(G_1)$ with $\norm{z_n}=1$ for all $n\in\N$ such that
  \ieq{
  g(\abs{s_n}) \norm{(is_n-G_1+G_2G_2^\ast)z_n}\to 0
  }
  as $n\to \infty$.
  For all $n\in\N$ we then also have
  \eq{
  0\leftarrow \re \iprod{g(\abs{s_n})(is_n-G_1+G_2G_2^\ast)z_n}{z_n} 
  = g(\abs{s_n}) \norm{G_2^\ast z_n}^2,
  }
  which further implies
  \eq{
  \MoveEqLeft \sqrt{g(\abs{s_n})} \norm{(is_n-G_1)z_n} 
  \leq g(\abs{s_n}) \norm{(is_n-G_1+G_2G_2^\ast)z_n} + \sqrt{g(\abs{s_n})} \norm{G_2^\ast z_n}
  \to 0
  }
  as $n\to \infty$.

  For each $n\in\N$ denote $m_n = \arg\min_{k}\abs{s_n-\gw_k}\in \Z$. Since $d=\inf_{k\neq l}\abs{\gw_k-\gw_l}>0$, we have that $\abs{s_n-\gw_k}\geq d/2$ for all $k\neq m_n$.  
  Denote $z_n=(z_n^k)\kZ$ where $z_n^k\in Y$.
  For any $0<\gd<1/2$ there exists $N_\gd\in\N$ such that 
  for all $n\geq N_\gd$ we have
  \eq{
  \gd^2
  &\geq g(\abs{s_n})\norm{(is_n-G_1)z_n}^2
  =g(\abs{s_n})\sum_{k\in\Z}  \abs{s_n-\gw_k}^2 \norm{z_n^k}^2\\
  &\geq g(\abs{s_n}) \abs{s_n-\gw_{m_n}}^2 \norm{z_n^{m_n}}^2 
  + g(\abs{s_n})\frac{d^2}{4}\sum_{k\neq m_n}  \norm{z_n^k}^2 .
  }
  Since $\abs{s_n}\to \infty$ as $n\to \infty$, we can assume that $\abs{\gw_{m_n}}\geq \gw_g+1$ for all $n\geq N_\gd$.
  Define $y_n = (y_n^k)_{k\in\Z}\in \ZI$ such that $y_n^{m_n}=z_n^{m_n}$ and $y_n^k=0$ for $k\neq m_n$. Then
  \begin{subequations}%
    \eqn{
    \label{eq:znminusznmnG1B1}
    g(\abs{s_n}) \norm{z_n-y_n}^2
    &= g(\abs{s_n}) \sum_{k\neq m_n}  \norm{z_n^k}^2 \leq \frac{4\gd^2}{d^2}\\
    \label{eq:znmnlowboundG1B1}
    \norm{z_n^{m_n}}^2 
    &=  \norm{z_n}^2 - \sum_{k\neq m_n}  \norm{z_n^k}^2 \geq 1-\frac{4\gd^2}{d^2}  .
    }
  \end{subequations}%
  The above estimates also imply 
  \eq{
  \gd^2 
  &\geq g(\abs{s_n}) \abs{s_n-\gw_{m_n}}^2 \norm{z_n^{m_n}}^2 
  \geq  \abs{s_n-\gw_{m_n}}^2  \left( 1-\frac{4\gd^2}{d^2} \right)
  }
  and thus $\abs{s_n-\gw_{m_n}}^2\leq \gd^2/(1-4\gd^2/d^2)\leq 2\gd^2$ if $\gd^2<d^2/8$ and $n\geq N_\gd$. 
  This means that the points $s_n$ approach the points in the set $\set{\gw_k}_{k\in\Z}$ as $n\to \infty$. Now
\eq{
 \sqrt{g(\abs{s_n})}\norm{\BI^\ast z_n}
&=\sqrt{g(\abs{s_n})}\norm{\BI^\ast y_n + \BI^\ast(z_n-y_n)} 
\geq \sqrt{g(\abs{s_n})}\norm{\BI^\ast y_n} - 
 \frac{2 \gd}{d} \norm{\BI} 
}
due to~\eqref{eq:znminusznmnG1B1}. Since 
$\abs{\ddb{\gw}\frac{1}{g(\gw)}}$ is bounded for $\gw>\gw_g$ and since $\abs{s_n-\gw_{m_n}}\to 0$ as $n\to \infty$, there exists $N_2\in \N$ such that $\frac{g(\abs{s_n})  }{g(\abs{\gw_{m_n}})}\geq 1/2$ for all $n\geq N_2$.
Using~\eqref{eq:znmnlowboundG1B1}
and $\norm{(G_{2m_n}^\ast)\pinv}^2=\norm{G_{2m_n}\pinv}^2\leq g(\abs{\gw_{m_n}})$ we get
\eq{
g(\abs{s_n}) \norm{\BI^\ast y_n}^2
\geq \frac{g(\abs{s_n}) }{\norm{(\BIn[m_n]^\ast)\pinv}^2} \norm{z_n^{m_n}}^2
\geq \frac{g(\abs{s_n})  }{g(\abs{\gw_{m_n}})} \left( 1-\frac{4\gd^2}{d^2} \right) 
\geq \frac{1}{2} \left( 1-\frac{4\gd^2}{d^2} \right)
}
for all $n\geq \max \set{N_\gd,N_2}$.
Combining the above estimates we see that for a small enough $\gd>0$ we thus have 
\eq{
\sqrt{g(\abs{s_n})}\norm{\BI^\ast z_n}
&\geq \sqrt{g(\abs{s_n})}\norm{\BI^\ast y_n} -
 \norm{\BI} \frac{2 \gd}{d}
\geq 
\frac{1}{\sqrt{2}} \left( 1-\frac{4\gd^2}{d^2} \right)^{\frac{1}{2}} -
 \norm{\BI} \frac{2 \gd}{d} >0
} 
for all $n\geq \max\set{N_\gd,N_2}$.
This contradicts the property $\sqrt{g(\abs{s_n})}\norm{\BI^\ast z_n}\to 0$ as $n\to \infty$, and therefore the proof is complete.
\end{proof}

\noindent\textit{Proof of Theorem~\textup{\ref{thm:CLnonunifstab}}.}
  The triangular structure of the operators $\hat{A}_e$ in the proofs of Theorems~\ref{thm:ContrOneMain},~\ref{thm:ContrOneROIM},~\ref{thm:ContrOneORP}, and~\ref{thm:ContrTwoMain} imply that if the function $g:\R_+\to[1,\infty)$ is chosen so that $\norm{(P_L(i\gw_k)K_{1k})\pinv}^2\leq g(\abs{\gw_k})$ for the controllers in Sections~\ref{sec:ContrOne}, \ref{sec:ContrROIM}, and~\ref{sec:ContrOneORP}, or so that $\norm{(G_{2k}P_K(i\gw_k))\pinv }^2\leq g(\abs{\gw_k})$ in the case of the controller in Section~\ref{sec:ContrTwo}, then
  \ieq{
  \norm{R(i\gw,A_e)} = O(g(\abs{\gw})).
  }
  Moreover, 
  $(P_L(i\gw_k)K_{1k})\pinv  = (P_L(i\gw_k)K_{1k})\inv $ in part (a), $(G_{2k}P_K(i\gw_k))\pinv =P_K(i\gw_k)\inv G_{2k}\inv$ in part (b), and $\norm{(P_L(i\gw_k)K_{1k})\pinv} = \frac{\norm{u_k}}{\abs{\gg_k}} \norm{P_L(i\gw_k)u_k}\inv$ in part (c).
  The 
  decay rate~\eqref{eq:CLnonuniformstate} follows from~\cite[Thm. 1.5]{BatDuy08}.

It remains to consider the behaviour of the regulation error. 
Assume  $v_0\in \Dom(S)$ and $x_{e0}= \xeoo-\Aemo\inv B_e v_0 $ where $\xeoo\in \Dom(A_e)$. Then $\Sigma Sv_0 = \Aemo\Sigma v_0 + B_ev_0$ implies
\eq{
\Aemo (x_{e0}-\Sigma v_0) 
= A_e\xeoo - B_ev_0 -\Sigma Sv_0 + B_e v_0
= A_e\xeoo  -\Sigma Sv_0 \in X_e
}
and thus $x_{e0}-\Sigma v_0\in \Dom(A_e)$.
Since $C_e$ is admissible
there exists $\kappa>0$ such that
$\int_0^1 \norm{\CeL T_e(s)x}ds\leq \kappa \norm{x}$ for all $x\in X_e$. 
Because $\Sigma$ is the solution of the regulator equations~\eqref{eq:regeqns},  the proof of Theorem~\ref{thm:ORP} and the nonuniform decay of $T_e(t)$ imply
\eq{
\hspace{6ex}
\int_t^{t+1} \norm{e(s)}ds
&=\int_0^1  \norm{T_e(s)T_e(t)(x_{e0}-\Sigma v_0)}  ds
\leq \kappa \norm{T_e(t)(x_{e0}-\Sigma v_0)}  \\
&\leq \frac{\kappa M_e}{\Mlog\inv(ct)} \norm{A_e(x_{e0}-\Sigma v_0)}
= \frac{\kappa M_e}{\Mlog\inv(ct)} \norm{A_e\xeoo-\Sigma Sv_0}.
\hspace{6ex}
\square
} 

\section{Robust Control of a Two-Dimensional Heat Equation}
\label{sec:heatex}

In this section we construct a controller that achieves robust output tracking and disturbance rejection for a two-dimensional heat equation 
\begin{subequations}%
  \label{eq:heatex}
  \eqn{
  x_t(\xi,t) &= \Delta x(\xi,t),  \qquad x(\xi,0)=x_0(\xi) \\[1ex]
  \pd{x}{n}(\xi,t)\vert_{\Gamma_1} &= u(t), \qquad
  \pd{x}{n}(\xi,t)\vert_{\Gamma_2} = d(t), \qquad
  \pd{x}{n}(\xi,t)\vert_{\Gamma_0} = 0, \\
  \qquad y(t) &=\int_{\Gamma_3}x(\xi,t)d\xi,   
  }
\end{subequations}%
on the unit square $\xi=(\xi_1,\xi_2)\in \Omega = [0,1]\times [0,1]$. Here $u(t)$ is the Neumann boundary control input and $d(t)$ is the external disturbance signal. The control and disturbance are located on the parts $\Gamma_1$ and $\Gamma_2$ of the boundary $\partial \Omega$, where $\Gamma_1 = \setm{\xi=(\xi_1,0)}{0\leq \xi_1\leq 1}$ and $\Gamma_2 = \setm{\xi=(0,\xi_2)}{0\leq \xi_2\leq 1/2}$ and the remaining part of the boundary is denoted by $\Gamma_0 = 
\partial\Omega\setminus (\Gamma_1\cup \Gamma_2)$.
The observation is on the part $\Gamma_3= \setm{\xi=(1,\xi_2)}{0\leq \xi_2\leq 1}$ of the boundary of the square.

The controlled heat equation can be written as an abstract linear system on $X=\Lp[2](\Omega)$ by choosing $A=\Delta$ with $\Dom(A)=\setm{x\in H^2(\Omega)}{\pd{x}{n}=0 ~ \mbox{on} ~ \partial \Omega}$ and choosing operators $B,B_d\in \Lin(\C,X_{-1})$ and $C\in \Lin(X_1,\C)$ such that
$B = \delta_{\Gamma_1}(\cdot)$, 
$B_d = \delta_{\Gamma_2}(\cdot)$, and
$Cx = \int_0^1 x(0,\xi_2)d\xi_2$~\cite{ByrGil02}.
We have from~\cite[Cor. 1]{ByrGil02} that the controlled heat equation~\eqref{eq:heatex} is a regular linear system with $D=0\in\C$.

We construct a robust controller using the method presented in Section~\ref{sec:ContrOne}. Our aim is to achieve output tracking of a continuously differentiable periodic reference signal (depicted in Figure~\ref{fig:heatoutput} in black) 
generated by an exosystem on $W=\lp[2](\C)$ with $S=\diag(ik)_{k\in\Z}$. If $(\phi_k)_{k\in\Z}$ denotes the canonical basis of $W$ and if we choose $F\in \Lin(W,\C)$ such that $F\phi_0=0$ and $F\phi_k = y_r(k) \abs{k}^{3/5}$ for $k\neq 0$,
where $y_r(k)$ are the complex Fourier coefficients of $\yref(t)$, then $\yref(t)$ is generated with the initial state 
$v_0 = (v_{0k})_{k\in\Z}$ with $v_{00}=1$ and $v_{0k}=\abs{k}^{-3/5}$ for $k\neq 0$.  The Fourier coefficients $y_r(k)$ of $\yref(t)$ satisfy $\abs{y_r(k)}=O(\abs{k}^{-3})$, and thus $\abs{F\phi_k} = O(\abs{k}^{-12/5})$.

\subsection{Stabilization and Controller Parameters}

Since $0\in\gs_p(A)$ the uncontrolled heat equation is unstable. If we choose $L_1\in \Lin(\C,X)$ and $K_2\in \Lin(X,\C)$ such that
$K_2 x = -\pi^2\int_\Omega x(\xi)d\xi$, 
and $L_1 = -\pi^2 \cdot\mathbf{1}$,
where $\mathbf{1}(\xi)= 1$ for all $\xi\in \Omega$,
then  $(A,\left[ \CL\atop K_2 \right],[B,~ L_1,~ B_d])$ is a  regular linear system and the semigroups generated by $A+L_1\CL$ and $(A+BK_2)\vert_X$ are exponentially stable.
In this example the transfer function of the plant has an explicit formula $P(\gl) = \frac{1}{\gl}$ for all $\gl\in \overline{\C_+}\setminus \set{0}$, and
$P_L(\gl)= (I-CR(\gl,A)L_1)\inv P(\gl)=(\gl+\pi^2)\inv$.

Since $\dim Y=1$, we choose the internal model to contain one copy of the exosystem so that $\ZI = W=\lp[2](\Z;\C)$ and $G_1=S$, and as suggested {Step~${4^\circ}$} and in Corollary~\ref{cor:ContrOnePolExp}, we choose $K_1\in \Lin(\ZI,\C) $ with components
\eq{
K_{1k} 
&= \frac{\gg_0}{1+\abs{k}^{1/2+\kappa}} \frac{P_L(i\gw_k)\inv}{\abs{P_L(i\gw_k)\inv}}  
= \frac{\gg_0}{1+\abs{k}^{1/2+\kappa}} \frac{ik+\pi^2}{\sqrt{k^2+\pi^4}},\\
G_{2k} 
& = -(P_L(i\gw_k)K_{1k})^\ast
= -\frac{\gg_0}{(1+\abs{k}^{1/2+\kappa}) \sqrt{k^2+\pi^4}} 
}
where $\kappa>0$ is fixed and small, and $\gg_0>0$. Then $\norm{(P_L(i\gw_k)K_{1k})\inv} = O(\abs{k}^{3/2+\kappa})$. Moreover, we define $L\in \Lin(\C,X_{-1})$ in such a way that 
\eq{
L &= L_1 + HG_2
= L_1+\sum_{k\in\Z} R(i\gw_k,\Amo +L_1\CL)BK_{1k}G_{2k}\\
&= -\pi^2 \cdot \mathbf{1}-\gg_0^2 \sum_{k\in\Z} \frac{R(i\gw_k,\Amo +L_1\CL)B}{(1+\abs{k}^{1/2+\kappa})^2(\pi^2-ik) }
}

\subsection{Solvability of the Robust Output Regulation Problem}

Since $\vspace{-1.6ex}\displaystyle\sup_{\gw\in\R}\norm{R(i\gw,A)}<\infty$, we can apply Corollary~\ref{cor:ContrOneAltEF}.
For all $k\neq 0$ 
\eq{
\abs{(P(i\gw_k)K_{1k})\inv} = \frac{\abs{k}(1+\abs{k}^{1/2+\kappa})}{\gg_0}
}
For our $\yref(t)$, $\norm{F\phi_k} = O(\abs{k}^{-12/5})$ and 
\ieq{
\abs{(P(i\gw_k)K_{1k})\inv} \norm{F\phi_k} = O(\abs{k}^{-9/10+\kappa}),
}
and thus $\left( \abs{(P(i\gw_k)K_{1k})\inv} \norm{F\phi_k} \right)_{k\neq 0}\in\lp[2](\C)$ and the output tracking of $\yref(t)$ is achieved whenever $0<\kappa<2/5$.

Since $(A,B_d,C)$ is regular, $\abs{P_d(i\gw_k)}\lesssim \abs{\gw_k}\inv= \abs{k}\inv$ for all $k\neq 0$, and
\eq{
\abs{(P(i\gw_k)K_{1k})\inv} \abs{P_d(i\gw_k)} \abs{\tilde{E}\phi_k}  
\lesssim \abs{k}^{1/2+\kappa} \abs{\tilde{E}\phi_k} .
}
By Corollary~\ref{cor:ContrOneAltEF} 
the controller
rejects all disturbance signals that can be expressed with $\tilde{E}\phi_k $ satisfying $ \abs{\tilde{E}\phi_k} \lesssim \abs{k}^{-\gb} $ 
for any exponent $\gb>1+\kappa$ and corresponding to the initial state $v_0 = \left( 
v_{0k}
\right)_{k\neq 0}$ with $v_{00}=1$ and
$v_{0k}=\abs{k}^{-3/5} $ for $k\neq 0$.   
This includes any $d(t)$ whose 
Fourier coefficients satisfy $\hat{w}(k) = \abs{k}^{-\tilde{\gb}}$ for any $\tilde{\gb}>8/5+\kappa$.

\subsection{Numerical Approximation and Simulation}

The controlled heat equation was simulated using a finite difference approximation with a $16\times 16$ grid on the square $\Omega = [0,1]\times [0,1]$. In the simulation the infinite-dimensional exosystem approximated using a $21$--dimensional truncation of the operator $S$.
The controller parameters were set to $\kappa = 1/8$ and $\gamma_0=12$.

For simulation the resolvent operators $R(i\gw_k,A+L_1\CL)$ appearing in $L$ were approximated numerically using a more accurate finite difference approximation with a $41\times 41$ grid, 
and the infinite sum was approximated with a truncation corresponding to the truncation of the exosystem.

The behaviour of the controlled system was simulated for  an external disturbance signal $d(t)= \cos(4t)+\frac{1}{2}\sin(t)$. Initial states of the plant and the controller are chosen to be zero.
Figures~\ref{fig:heatoutput} and~\ref{fig:heaterrints} 
depict the output $y(t)$ of the controlled plant on the interval $[4\pi,12\pi]$ and the behaviour of the integrals $
\int_t^{t+1}\norm{e(s)}ds$, respectively. 
Finally, Figure~\ref{fig:heatsurf} depicts the behaviour of the state of the controlled system on $\Gamma_3$.

  \begin{figure}[ht]
\begin{minipage}{0.48\linewidth}
    \begin{center}
      \includegraphics[width=1.06\linewidth]{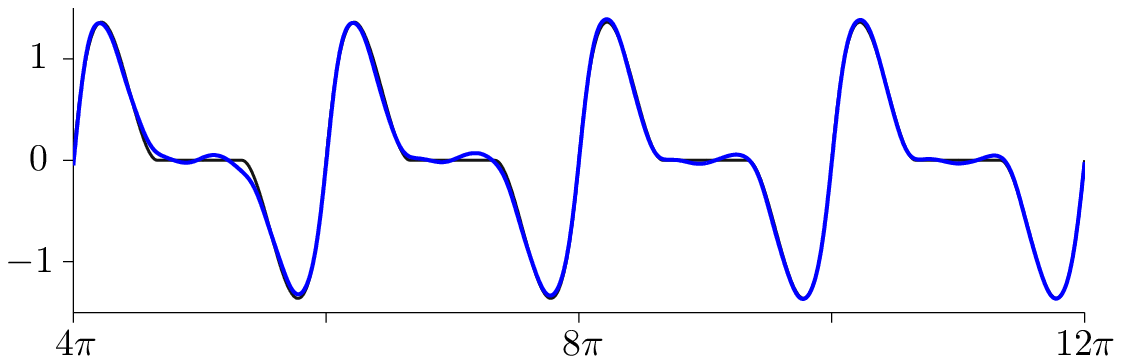}
      \end{center}
      \caption{Output $y(t)$ of the controlled plant.}
      \label{fig:heatoutput}
  \end{minipage}
  \hfill
  \begin{minipage}{0.48\linewidth}
      \begin{center}
	\includegraphics[width=\linewidth]{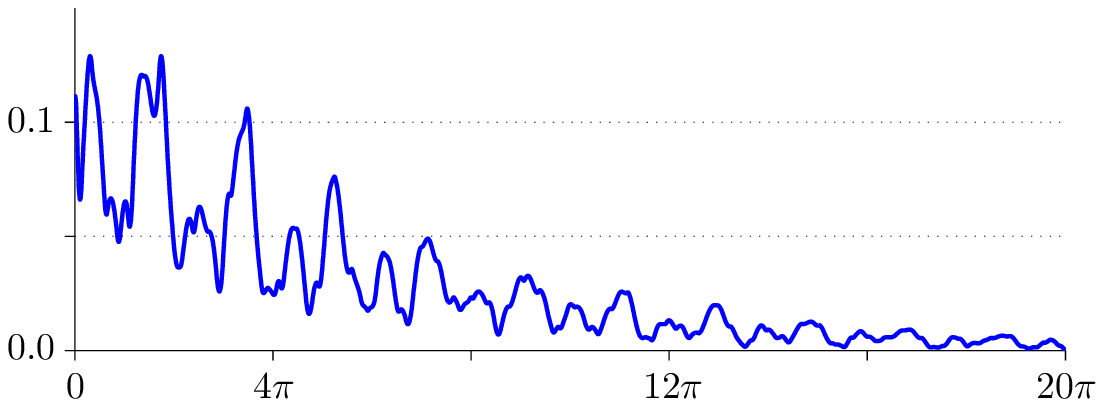}
	\end{center}
	\caption{
	Behaviour of the error integrals.
	}
	\label{fig:heaterrints}
    \end{minipage}
      \end{figure}

\begin{figure}[ht]
  \begin{center}
    \includegraphics[scale=0.6]{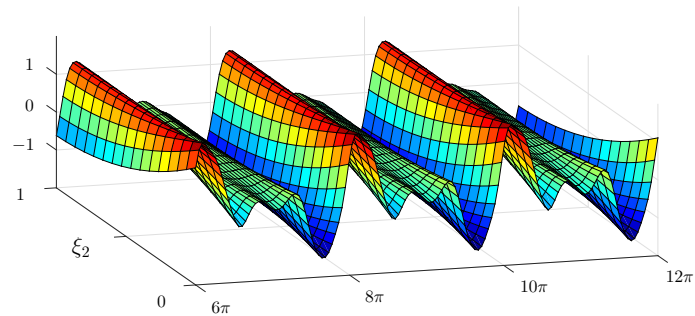}
\label{fig:heatsurf}
\caption{State of the controlled heat equation on $\Gamma_3$.}
\end{center}
\end{figure}

\end{document}